\documentclass[1p]{elsarticle1}

\usepackage{lineno,hyperref}
\usepackage{amssymb}
\usepackage{eucal}
\newcommand{\R}{{\Bbb R}}

\newcommand{\N}{{\Bbb N}}

\usepackage{xcolor}
\definecolor{awesome}{rgb}{1.0, 0.13, 0.32}
\definecolor{darkblue}{rgb}{0.0, 0.0, 0.55}
\usepackage{physics}
\usepackage{amsmath}
\usepackage{tikz}
\usepackage{mathdots}
\usepackage{yhmath}
\usepackage{cancel}
\usepackage{color}
\usepackage{siunitx}
\usepackage{array}
\usepackage{multirow}
\usepackage{amssymb}
\usepackage{gensymb}
\usepackage{tabularx}
\usepackage{extarrows}
\usepackage{booktabs}
\usetikzlibrary{fadings}
\usetikzlibrary{patterns}
\usetikzlibrary{shadows.blur}
\usetikzlibrary{shapes}

\usepackage{amsmath}

\newtheorem{thm}{Theorem}
\newtheorem{lemma}[thm]{Lemma}
\newtheorem{corollary}[thm]{Corollary}

\newtheorem{definition}[thm]{Definition}
\newtheorem{remark}[thm]{Remark}
\newtheorem{example}[thm]{Example}
\newproof{proof}{Proof}
\newcommand\mL{L\kern-0.08cm\char39}
\definecolor{crimson}{rgb}{0.86, 0.08, 0.24}
\definecolor{burgundy}{rgb}{0.5, 0.0, 0.13}
\begin{document}

\begin{frontmatter}

%\title{A simple approach to the wave uniqueness problem for  monostable  delayed reaction-diffusion equations} 

\title{The peak-and-end rule and differential equations with maxima:  a view on the unpredictability  of happiness}

\author[a]{Elena Trofimchuk}
\author[b]{Eduardo Liz}
\author[c]{Sergei Trofimchuk\corref{mycorrespondingauthor}}
\cortext[mycorrespondingauthor]{\hspace{-1mm} {\it e-mails addresses}: trofimch@imath.kiev.ua (Elena Trofimchuk); eliz@uvigo.es (Eduardo Liz);   trofimch@inst-mat.utalca.cl (Sergei Trofimchuk, corresponding author). \\}
 \address[a]{Department of Differential Equations, Igor Sikorsky Kyiv Polytechnic Institute, Kyiv, Ukraine}
\address[b]{Departamento de Matem\'atica Aplicada II, 
Universidade de Vigo, 36310 Vigo, Spain}
\address[c]{Instituto de Matem\'atica, Universidad de Talca, Casilla 747,
Talca, Chile }

\bigskip

\begin{abstract}
\noindent  In the 1990s, after a series of experiments,   the behavioral psychologist  and  economist    Daniel Kahneman and his colleagues formulated  the following 
Peak-End evaluation rule: {\it  the remembered utility of pleasant or unpleasant episodes is accurately predicted by averaging the Peak (most intense value) of instant utility (or disutility) recorded during an episode and the instant utility recorded near the end of the experience}
(D. Kahneman et al.,  1997,  QJE,  p. 381). Hence, the simplest mathematical model for time evolution of the  experienced utility function $u=u(t)$ can be given by the  scalar  differential equation $u'(t)=a u(t) + b \max \{u(s) : s\in [t-h,t]\}+f(t) \ (*),$
where $f$ represents exogenous stimuli, $h$ is the maximal duration of the experience, and $a,b \in \R$ are some averaging weights. In this work, 
we  study equation $(*)$ and  show that,  for a range of parameters $a, b, h$ and  a periodic sine-like term $f$, the dynamics of $(*)$ can be completely
described in  terms of an associated one-dimensional
dynamical system generated by a piece-wise continuous map
from a finite interval into itself. We illustrate our approach with  two examples. In particular, we show that  the hedonic utility $u(t)$  (`happiness')  can
exhibit chaotic behavior. 
\end{abstract}

\end{frontmatter}

\centerline{\sc To the memory of Anatoly Samoilenko (1938-2020)}

\newpage

%\linenumbers

\section{Introduction}
\noindent  In the 1990s, after a series of experiments,  the behavioral psychologist  (and Nobel laureate in economics)   Daniel Kahneman with his colleagues formulated  the following 
Peak-End evaluation rule: {\it  the remembered utility of pleasant or unpleasant episodes is accurately predicted by averaging the Peak (most intense value) of instant utility (or disutility) recorded during an episode and the instant utility recorded near the end of the experience} (excerpt from \cite{PER2}, p. 381).  This rule obtained multiple applications (including   customer service,  price setting strategies, medical procedures,  education etc) and nowadays, the Peak-End theory 
has become one of  the active areas of research in the field of behavioral science, e.g. see \cite{CR, PER1,PER2,PER3,PER4,PZ} and references therein. 

Accordingly, the simplest mathematical model for time evolution of the  experienced utility function $u=u(t)$ can be given by the  scalar  differential equation
\begin{equation}\label{meq}
u'(t)=a u(t) + b \max \limits_{s\in [t-h,t]}u(s) +f(t)
\end{equation}
where $f$ represents exogenous stimuli, $h>0$ is the maximal duration of the experience and $a,b$ are some real coefficients. The determination of  qualitatively plausible psychological 
parameters $f, a, b$ seems to be a rather difficult task (which we do not address here); on the other hand, it is  reasonable to consider the case when  $f$ is a  $T$-periodic continuous sine-like function (see Definition~\ref{sine-like} below), assuming that  the rate of change of the utility $u(t)$ is affected by linear decay (with coefficient $\alpha >0$) and  is proportional (with coefficient $\beta >0$)  to the difference between its instant value and its peak on the precedent fixed time interval:  
\begin{equation} \label{meqo}
u'(t)=-\alpha u(t) + \beta  (u(t)-\max \limits_{s\in [t-h,t]}u(s)) +f(t).
\end{equation}
Akin evolutionary rules, with the term $u(t-h)$ instead of $\max \{u(s):s\in [t-h,t]\}$, can be found in other comparable cases: see, for instance,  the celebrated Kalecki difference-differential equation describing a macroeconomic model of business cycles \cite{RF,K35,AAK} or the mathematical model of emotional balance dynamics proposed in \cite{TB}.  The works \cite{BEW, Er} show how the psychology of agents trading the foreign currency generates a similar dynamical mechanism expressed by the equation 
 $$
 u'(t)=-b|u(t)|u(t)+ a(u(t)-u(t -1)), \quad a,b >0. 
 $$
Comparing equations (\ref{meq}) and (\ref{meqo}), we obtain that $a = \beta - \alpha,\  b =- \beta$. In this way, the situation when $b <0$ and $a+b <0$ might appear as  more interesting from the applied point of view, and, as we  manifest in the present paper, it is certainly more interesting by its mathematical implications. 
In particular,  we will show that for a range of parameters $a, b, h$ and a periodic  term $f$, the dynamics of  (\ref{meq})  can exhibit chaotic behavior.

Even assuming that equation (\ref{meqo})  is a phenomenological model, our study could be considered as another attempt to use mathematics to understand the behaviour of happiness, a topic that goes back at least to Edgeworth's calculus of pleasure or ``hedonimetry" in 1881 \cite{Edge}.
Indeed,  `Happiness' \cite{G} is one of the possible interpretations of the experienced utility, and, from their own individual experience, everyone knows that 
happiness is unpredictable \cite{JG}. Remarkably, there exist well-documented descriptions of visibly  chaotic time evolution of  happiness \cite{th}, see Figure~\ref{Fig3b} and compare it with a numerical solution  obtained for a particular case of (\ref{meq}) in Figure~\ref{Fig3}.

 \begin{figure}[htb] 
\centering 
\includegraphics[width=10cm]{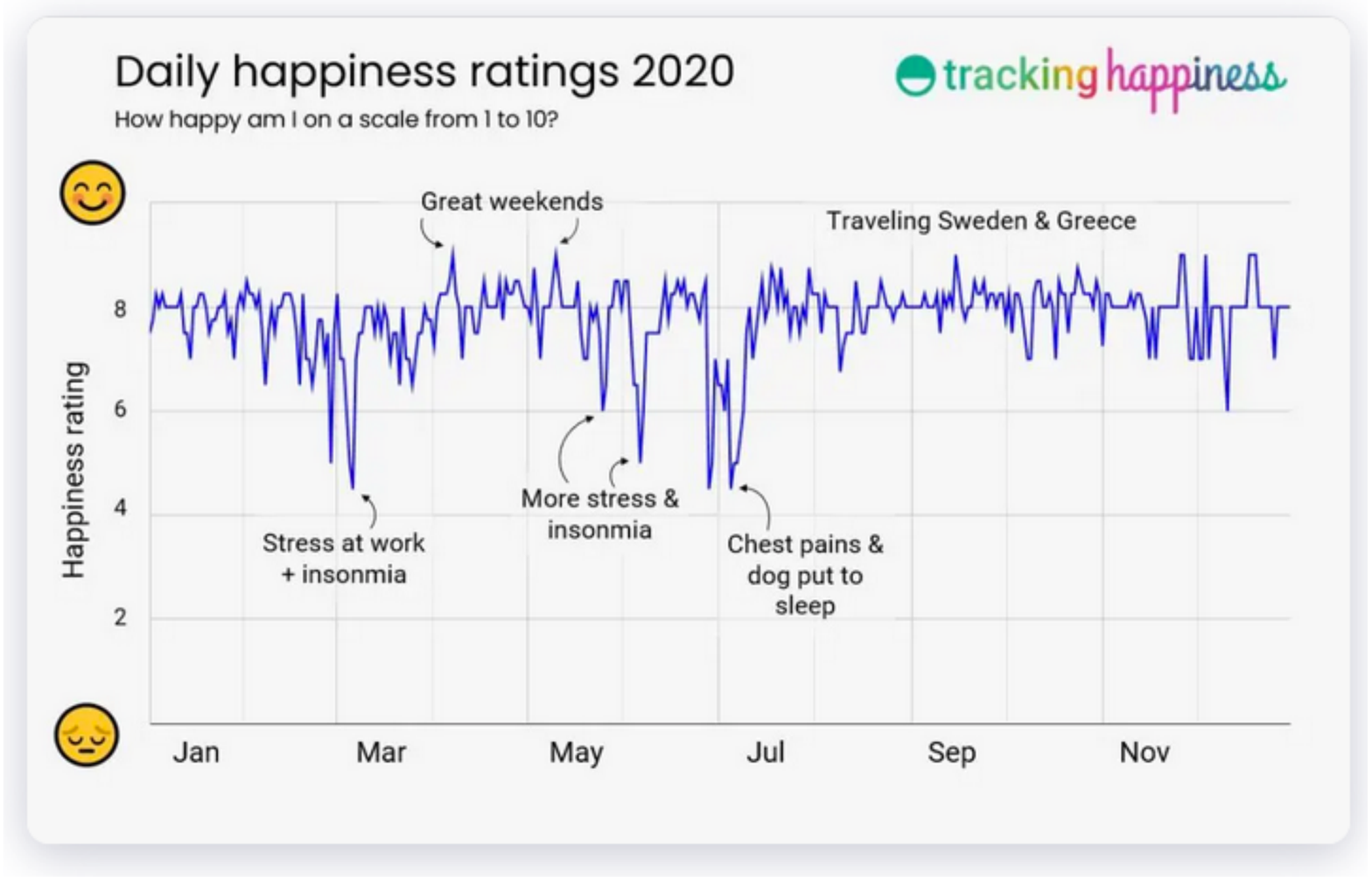}
\caption{\hspace{0cm}   The 2020 Happiness review  by H. Huijer  taken from \cite{th}. The figure is published with the kind permission of Hugo Huijer.} 
\label{Fig3b}
\end{figure}

In any event, the main goal of our studies is the elaboration of a satisfactory mathematical framework to deal with 
the quasilinear functional differential equation (\ref{meq}). As far as we know, the first article dedicated to equations with maxima  appeared in 1964 \cite{Pe} and in the survey \cite[Section 12]{my} on the theory
of functional differential equations, A. Myshkis
singled out systems with maxima as differential
equations with deviating argument of complex structure. Particularly he noted that {\it`the specific character of these questions is not yet sufficiently clear'} \cite[p. 199]{my}.  Denote by $C[-h,0]$ the set of continuous functions from $[-h,0]$ to $\R$. We notice that  the functional $f: \R \times C[-h,0] \to \R$ defined by $f(t,\phi)= a\phi(0) + b \max\{\phi(s):s\in [-h,0]\} +f(t)$, which corresponds to  the right-hand side of (\ref{meq}), is globally Lipshitzian in $\phi$, 
which guarantees the existence, uniqueness, global continuation and continuous dependence on initial data of the solutions to (\ref{meq}). However, this functional is not differentiable in $\phi$. By using the representation $\max\{\phi(s), s\in [-h,0]\} = \phi(-\tau(\phi))$ for some $\tau(\phi) \in [0,h]$, we see that (\ref{meq}) can be considered as a functional differential equation with state-dependent delay. Note that function $\tau: C[-h,0] \to [0,h]$ is clearly discontinuous at each constant element.\footnote{Even though $\tau(\phi)$ is not uniquely defined.}    The above considerations show that an appropriate functional space for  the evolutionary system (\ref{meq})  should be the space of continuously differentiable functions $C^1[-h,0]$ instead of $C[-h,0]$, cf. \cite{PT}. 

We will call (\ref{meq}) the Magomedov equation, in honor of the Azerbaijani mathematician who introduced this model in the late 70s 
and  since then has analyzed several  particular cases of it with periodic forcing term $f(t)$, see  \cite{BH,17,Mago,STB,SB}.  In his   monograph \cite{Mago} dedicated to equations with maxima, Magomedov explains how the periodic equation (\ref{meq}) can be used 
for modeling automatic control of voltage in a generator of constant current, see   \cite[pp. 4--7]{Mago}. 

Besides the above mentioned applications, the periodic equation (\ref{meq}) plays an important role in the stability theory for the  
delay differential equation
$$u'(t) = au(t)+ bf(t, u_t),$$
where $a, b<0$, $u_t(s)=u(t+s),\  s\in[-h,0]$, and the continuous functional $f: \R\times C[-h,0] \to \R$ satisfies either 
the following (sublinear) Yorke condition \cite{Y}
\begin{eqnarray}
-\max_{s \in [-h,0]}(- \phi(s)) \leq f(t, \phi) \leq  \max_{s \in [-h,0]} \phi(s), \
t \geq 0, \ \phi \in C[-h,0],
\label{73}
\end{eqnarray}
or its generalized (nonlinear) version introduced in \cite{ltt}. In this context,  model (\ref{meq}) is used as a key test 
equation whose analysis determines the optimal stability regions for equations satisfying  one of the  aforementioned Yorke conditions. 
For instance, in the simplest situation when $a=0$,  
equation (\ref{meq}) has a uniformly asymptotically stable 
periodic solution for every periodic function $f(t)$ if and only if $0 < -bh < 3/2$
(that constitutes a variant of the so-called Myshkis-Wright-Yorke 3/2-stability criterion, see  \cite{FLOT,kr, YW,PT}). 

Among other mathematical objects closely related to equation (\ref{meq}), we would like to mention the Hausrath equation
$$u'(t)= b(\max_{s \in [t-h,t]}|u(s)|-u(t)),$$
analyzed in \cite[pp. 73-74]{HD}  and   the Halanay  inequality 
$$u'(t)\leq a u(t)+b\max_{s \in [t-h,t]}u(s)$$
which  became an important tool in the stability theory of 
functional differential equations, see \cite{BT,17,ILT,H2} for the further references.

The present work extends previous studies  \cite{17, PT} where, in particular,
the existence of multiple periodic solutions to equation  (\ref{meq}) was established by 
using Krasnoselsky's rotation number and introducing a  substitute of the variational equation for  the non-smooth model (\ref{meq}). 
Our approach in this paper is cardinally different,  its workhorse is  an associated 
selfmap $\cal R$ of an interval called `the return map' in the paper. This function
allows us to reproduce the sequence of consecutive `qualified'  local maxima $q_j>q$ (i.e. having the property 
$u(q_j,p) = \max_{s \in [q_j-h,q_j]}u(s,p)$) of 
each solution $u(t,p)$ to (\ref{meq}) with initial condition $u(s,p)=p,$ $s\in[q-h,q]$ (actually, we will define $\cal R$  by ${\cal R}(p) = u(q_1,p)$). As we will show, the information stored in $\cal R$ is well enough to  describe the dynamics in (\ref{meq}). Now, analysing 
the dependence of the `qualified'  local maximum $u(q_1,p)$  on $p$, one can observe that at some specific values of 
$p$ this maximum disappears due to a cusp catastrophe.  
Accordingly, the return map $\cal R$ has a discontinuity  at each such point so that important efforts in Section 2 are focused on  the studies of the continuity and differentiability properties of  $\cal R$. 
In particular, while computing the derivative  ${\cal R}'(p)$, we have found another interpretation of the  aforementioned variational equation for (\ref{meq}).

Finally, in Section 3 we  show that, in spite of  the uniqueness of $T$-periodic solution to
(\ref{meq}) for all sufficiently
small and large values of $hT^{-1}$, in general 
equation  (\ref{meq}) possesses a global attractor ${\cal A}$ 
with rather complicated dynamical structure.\footnote{Curiously, the first working hypothesis  about equation (\ref{meq}) was that, due to the positive homogeneity 
of the $\max$-functional,  this equation  has a unique periodic solution for all choices of $a+b\not=0$ and periodic functions $f(t)$. Thus,   the possibility of complicated dynamics in  (\ref{meq}) was quite surprising  for the authors.}
Indeed, for a wide range of parameters $a, b$, the restriction of the map $\cal R$ on an appropriate compact subset of its continuity domain has a generalised horseshoe. 
Precisely this fact implies the existence of an infinite number of different periodic solutions to (\ref{meq}) as well as sensitive dependence on the initial values (chosen in some subset of continuous functions).
Our  example in Subsection \ref{sec:chaos} extends a relatively small set of delay differential equations coming from applications where 
the existence of `chaotic' behaviour has been proved analytically, cf. \cite{walther} and its references.   As usual, this requires elementary but laborious evaluations of some auxiliary smooth 
functions on compact sets. This work is realised in an Appendix. 

\section{Associated one-dimensional dynamics}

\subsection{Some properties of the solutions to (\ref{meq})}
\label{sec:prop}

For $a+b \not=0 $, let us consider the following family of  initial value problems
for periodic functional  differential equations:
\begin{eqnarray}
u'(t)=  a u(t) + b\max_{s\in[t-h,t]} u(s) + f(t+\tau), \ \tau \in \R,
\label{1}\\ \noalign{\medskip}
u(s+\tau)= \phi(s),  s \in [-h,0], \ \phi  \in C:=C[-h,0].
\label{2}
\end{eqnarray}
If $T$ is the minimal period  of $f$, it suffices to
consider values  $\tau \in [0,T)$. Identifying the points $0$ and $T$, we
can replace this interval with the circle $S^1$.  This means that, once $a, b, f$ are fixed in (\ref{1}),  we
can identify each pair (\ref{1}), (\ref{2})
with the point $x=(\tau, \phi)$ from the phase space
$X = S^1 \times C$.

Let $u(\cdot,\tau, \phi):[\tau - h, +\infty) \to \R$
be the solution to (\ref{1}),(\ref{2}). For every $\mu \geq 0$  and $(\tau, \phi) \in X$,
we consider the function $\psi(s) = u(\tau + \mu +s ,\tau, \phi),\, s\in[-h,0],$ and the
representation $\tau + \mu = \tau_1 + kT, \ \tau_1 \in [0,T)$.  For each $\mu\in\R_+$, we define the application
$F^{\mu}: X \to X$ by $F^{\mu}(\tau, \phi)=(\tau_1, \psi)$.
By  definition, $\tau_1 = (\tau + \mu) \pmod{T}$, $F^0 = Id$ and $F^{\mu}\circ
F^{\nu}= F^{\mu+ \nu}$ for all $\mu, \nu \geq 0$.  Here, for $x\in\R$, we 
define $x \pmod{T} \in [0,T)$ as the unique real number in $[0,T)$
such that $x = x\pmod{T} + kT$ for some 
integer $k$.

Moreover, the continuous
dependence of the solution  $u(t,\tau, \phi)$ on parameters $(\tau, \phi)$
implies that the map $F: X \times \R_+ \to X$ defined by $F(\tau, \phi,\mu)=F^{\mu}(\tau, \phi)$ is  continuous. Hence,
$F$ determines a skew-product  semidynamical system with
$C$ as the fibre space and $S^1$ as the base space.

In this subsection we  identify  a  subset of parameters $(a,b) \in \R^2$
for which $F^{\mu}$ has a compact global attractor ${\cal A}(F)$ (i.e.,  
a compact invariant connected subset of $X$ attracting every 
trajectory of the dynamical system). In view of  J. Hale's general theory in  
\cite{HMO},  ${\cal A}(F)$
attracts all bounded sets of $X$ and ${\cal A}(F) = \cap_{\mu \geq 0}F^{\mu}{\cal K}$
where ${\cal K}$ is any compact set which attracts all compact sets of $X$.
Note that the case $f\equiv0$ was  already studied in \cite{PT}, where
a criterion for the equality
${\cal A}(F) = S^1 \times \{0\}$ was  established.

We will consider sine-like $T$-periodic functions in the sense of the following definition:
\begin{definition}
\label{sine-like}
We say that a T-periodic continuous function $f:\R\to\R$ has sine-like shape if
there exist $t_0, t_1$ such that $0<t_1-t_0<T$, $f$ is strictly monotone on $[t_0,t_1]$ and on $[t_1,t_0+T]$, and $t_1$ is a turning point of $f$.
\end{definition}

Our proof of the existence
of a compact global attractor is based on the following lemmas
describing some properties of the solutions to equation  (\ref{meq}):
\begin{lemma}\label{L1} Assume that $h <T$,  $a+b \not=0 $ and the T-periodic continuous function $f$ has sine-like shape.
Let 
$u: [\alpha, + \infty) \to \R$ be 
a solution to (\ref{meq}). Then at least one of the following
options is satisfied:
\begin{enumerate}
\item[1)] there exists $\tau > \alpha$ such that $u$ strictly
increases on $[\tau, +\infty)$ and  $u(t) \to + \infty$ as
$t \to +\infty$;
\item[2)] there  exists $\tau_1 > \alpha + h$ such that
$U(t) := \max\limits_{s \in [t-h,t]}u(s)$ decreases
on $[\tau_1, +\infty)$ and $u(t) \to - \infty$ as
$t \to +\infty$;
\item[3)] there exist $\tau_2 > \alpha +h$ and $\varepsilon > 0$
such that
$\max\limits_{s \in [\tau_2-h,\tau_2 + \varepsilon]}u(s) = u(\tau_2)$.
\end{enumerate}
\end{lemma}
\begin{proof} Consider  the function
$U:(\alpha + h, + \infty) \to \R$
defined in 2). We have the following three alternatives: 

(I)  $U$ is decreasing on some  interval $(\sigma, +\infty)$.  If, in addition  $U(+\infty) =-\infty$, then the second option of the lemma is satisfied.  So, suppose that $U(+\infty) = U_*$ is finite.  Then $u$ satisfies 
the differential equation 
$$
u'(t)=a u(t) + bU_* +b g(t) +f(t),
$$
where $g(t):=U(t)-U_* \geq 0$ for $t \geq \sigma$, $g(+\infty)=0$. 

If, in addition,  $a=0$, then $b\not=0$ and 
$$
u(t) = u(\sigma) + \int_\sigma^t (f(s)-\bar f)ds +b \int_\sigma^t (U_* +b^{-1} \bar f +g(s))ds, \quad \bar f:= T^{-1}\int_0^Tf(s)ds. 
$$
Since,  together with $U$, the solution $u$ is  bounded on $[\sigma, +\infty)$, we obtain   that $U_* +b^{-1} \bar f =0$ and 
\begin{equation}\label{rep}
u(t) = p_1(t) + g_1(t), \ t \geq \sigma, 
\end{equation}
 where  
 $$g_1(t)=C+b \int_\sigma^t g(s)ds$$
 is a bounded monotone function. Clearly, we can choose the real number $C$ in such a way that $g_1(+\infty)=0$.
 Moreover,
$$p_1'(t) = f(t) - f(\theta), \quad \min_{s\in\R} f(s) < f(\theta) < \max_{s\in\R} f(s), $$
 for some fixed $\theta\in [0,T]$, so that $p_1$ has exactly two critical points on each half-closed interval of length $T$. 
Thus $p_1$ is a sine-like $T$-periodic function.  However, since $h <T$, this implies 
that $U(t)$ can not be  monotone, a contradiction. 

Consider now the case when $a\not=0$. Similarly, we find that representation (\ref{rep}) is true in this situation, 
with  $g_1(+\infty)=0$ and $p_1$ being the unique $T-$periodic solution of the equation 
\begin{equation}\label{meqs}
x'(t)=a x(t) + f_1(t), \quad f_1(t):=  bU_*  +f(t).
\end{equation}

This will produce again a contradiction once  it is  established that the $T-$periodic function $p_1$ is sine-like. First consider $a<0$, 
then 
$$
-\frac{f_1(T_2)}{a}=\int^t_{-\infty}e^{a(t-s)} f_1(T_2)ds<  p_1(t)=\int^t_{-\infty}e^{a(t-s)}f_1(s)ds <  -\frac{f_1(T_1)}{a},
$$
where $f_1(T_1)= \max_\R f_1(s), \ f_1(T_2)= \min_\R f_1(s)$ for some $T_1<T_2<T_1+T$.  Therefore the graph of the solution $x$
belongs to the rectangle $[T_1,T_1+T] \times (-a^{-1}f_1(T_2), -a^{-1}f_1(T_1))$ of the extended phase plane. The zero isocline for 
(\ref{meqs}) is given by the graph of $x= -a^{-1}f_1(t)$.  In the open region  below this isocline, the solutions of (\ref{meqs}) are increasing, while they are decreasing above the
zero isocline.  Take any point $P_s=(s,-a^{-1}f_1(s))$ for $s \in (T_1,T_2)$; it is easy to see that each trajectory of (\ref{meqs}) through $P_s$ is strictly decreasing in the backward direction and therefore has a unique intersection with the zero isocline on $[T_1,s]$.    This proves that $x=p_1(t)$ has a unique intersection with $x= -a^{-1}f_1(t)$ on 
the interval $[T_1,T_2]$ (say, at some point $s_*\in (T_1,T_2)$), it is strictly increasing on $[T_1,s_*]$ and  strictly decreasing on some maximal interval $[s_*,s^*]$, 
where $s^* \in (T_2,T_1+T)$ and $p_1(s^*) = - a^{-1}f_1(s^*)$. By the same argument as before, we obtain that $x(t)=p_1(t)$ cannot cross the zero isocline for $t \in (s^*,T_1+T]$
a second time and therefore $p_1$ is strictly increasing on $[s^*,T_1+T]$. This means that $p_1$ has sine-like form.  

To complete the analysis of the first alternative, we should consider  $a>0$. This case can be easily reduced to the previous one since the periodic function $q(t):=p_1(-t)$ satisfies the equation 
$
q'(t)= - aq(t) - f_1(-t)
$
so that  $q$ has sine-like shape. \\

(II)  Next, we consider  the alternative when $U$ is increasing on some  interval $(\sigma, +\infty)$. Evidently, if $U$ is  eventually 
strictly increasing, then $U(t) = u(t)$ for all
sufficiently large values of $t$. This implies that $u$
satisfies the equation $u'(t) =  (a + b)u(t) + f(t)$. However, as we have seen in (I),  the described situation
can occur only if $a+b>0$, with $u(t) \to + \infty$.  This is the first option in the statement of  Lemma \ref{L1}. 

So assume that $U$ is increasing on $(\sigma, +\infty)$ and there exists a sequence of maximal intervals $[a_j,b_j]$, $a_j< b_j < a_{j+1}, \ \lim a_j =+\infty$, such that 
$U$  is constant  on each of them. Then clearly the third option of the lemma is satisfied for each $\tau_2=a_j$. \\

(III) Finally,  if $U$ is not eventually monotone then there exist $\alpha +2h < s_1<s_2<s_3$ such that $U(s_1) < U(s_2) > U(s_3)$. If $\hat s$ is the leftmost point where the absolute 
maximum of $u(t)$ on the interval $[s_1,s_3]$ is attained, then $\hat s\in (s_1,s_3)$  and  the third option of the lemma is satisfied with $\tau_2=\hat s$.

This completes the proof of Lemma \ref{L1}.\qed 
\end{proof}
\begin{lemma} \label{L2} Assume that the trivial solution of the delay equation
\begin{equation}
u'(t)= a u(t) + b\max_{s \in [t-h,t]}u(s)
\label{eq}
\end{equation}
is uniformly asymptotically stable,  that is, either of the following conditions holds \cite{VB}:
\begin{equation}
  b + a<0,\ ah \leq 1 \quad \mbox{or} \quad bh < -\exp(a h -1),\ ah \geq  1.
\label{ga}
\end{equation}

Then every solution of (\ref{1})
is bounded on each interval $[r, + \infty)$, $r \in \R$, belonging to  its domain. 
\end{lemma}
\begin{proof} 
 First, we notice that, by  \cite[Theorem 2.1]{PT},  every non-trivial solution of  (\ref{eq}) is eventually strictly monotone. This implies that  
the zero solution to (\ref{eq}) is uniformly exponentially stable if and only if  the 
characteristic function $z -a -b e^{-z h}$  associated to the linear delay-differential equation 
\begin{equation}\label{lde}
u'(t)= au(t)+bu(t-h)
\end{equation}
 does not have nonnegative real zeros (hence, the exponential stability of equation (\ref{lde})
 implies   the uniform exponential stability of (\ref{eq})). It can be proved that this property holds if and only if either of conditions in (\ref{ga}) holds,
 a stability result established in  \cite{VB}. 
 
In the following, we assume that equation (\ref{eq}) is uniformly exponentially stable.
Then, for every solution $v: \R_+ \to \R$ of (\ref{eq}) there is a real number $\mu$ such that  $\mu >2h > 0$  and
$\|v_{\mu}\| \leq 0.5 \|v_{2h}\|$, where $\|\phi\|=\max\{|\phi(s)|:s\in[-h,0]\}$, $v_d(s)= v(d+s), s\in[-h,0]$. 

Assume that there is an unbounded solution  $u$  of (\ref{1}). Then there 
exists  a sequence $t_n \to +\infty$ such that
$|u(t_n+\mu)|= \max\limits_{r \in [t_n,t_n+\mu]}|u(r)| \to \infty$
as $n \to \infty$. The sequence $\{v^{(n)}\}$ defined by
$$v^{(n)}(t)=\frac{u(t+t_n)}{\max\limits_{s \in [t_n,t_n+\mu]}|u(s)|}\, ,\; t\in[0,\mu],$$
is relatively compact in $C([0,\mu])$. Then,  there exists a subsequence   $\{v^{(n_j)}\}$ that converges uniformly to a function $v \in C([0,\mu])$. Finally, $v$
satisfies equation (\ref{eq}) on the interval $[h,\mu]$ and 
$1= |v(\mu)|=\|v_{\mu}\| = \max\limits_{r\in[0,\mu]}|v(r)| \geq \|v_{2h}\|$,
a contradiction with the definition of $\mu$. \qed
\end{proof}
\subsection{Construction of the return map}
\label{construction}
In the sequel,  we  assume that either of the two conditions in (\ref{ga}) holds, that $h <T$, and consider  $T$-periodic continuous functions $f$ with sine-like shape. Set also  
 \begin{equation}
\label{ftilde}
\tilde f(t) =\frac{f(t)}{|a +b|}.
\end{equation}
Notice that (\ref{ga}) implies that $a+b<0$.

After appropriate
change of variables  $u \to
v + \min_{t \in \R} \tilde f(t)$
and $t \to s + const$,   without loss of generality, we can assume that the following condition holds:\\

\noindent ({\bf H}) $f$ is a continuous $T-$periodic function, strictly decreasing on the
interval $I_1 = [0,\beta]$ and strictly increasing on  $I_2 =[\beta, T]$, with $\min_{t \in \R} f(t) =0.$\\

Clearly, if  $p \in \tilde f(I_1)$ then $p = \tilde f(q)$ for a unique
$q \in I_1$. Let $u(\cdot,p): [q, + \infty) \to \R$ be the solution
of the initial value problem $u(s,p) \equiv p, \ s \in [q-h,q]$ for equation (\ref{meq}). Then Lemmas \ref{L1} and \ref{L2} guarantee the existence of 
 $\nu = \nu (q) > q$ and $\varepsilon > 0$ such that
\begin{equation}
\label{condition}
u(\nu,p)= \max\limits_{r \in[\nu - h,\nu+\varepsilon]} u(r,p).
\end{equation}

Let $\nu^*$  be the  the smallest $\nu>q$ satisfying (\ref{condition})  and set  ${\cal R}(p) = u(\nu^*,p)$. We refer the reader to Figure~\ref{FigSc}  below for an illustration of the definition of ${\cal R}$ and some characteristic points involved in our results.
The next statement says that ${\cal R}(\tilde f([0,\beta])) \subset \tilde f([0,\beta))$, 
 in other words, that   ${\cal R}(p) >0$ for each $p \in  \tilde f([0,\beta])$ and  the application 
 \begin{equation}\label{R} 
 {\cal R}: \tilde f([0,\beta]) \to \tilde f([0,\beta])
\end{equation}
 is well defined. 
\begin{lemma}\label{Lr}{ Let $u: [-h, + \infty) \to \R$ be a solution
of (\ref{meq}),   and let  $\tau > 0$ be a point of local maximum  for
$u$; moreover, assume that, for some $\epsilon >0$,  $u(\tau) \geq u(t)$ for all $t \in [\tau - h, \tau+\epsilon)$.
Then $u(\tau) = f(\tau)/|a + b|$ and
$\tau ^*=\tau \pmod{T} \in [0,\beta)$.}
\end{lemma}
\begin{proof}The first conclusion of the lemma is evident.  We prove the
second one by contradiction. Suppose that
$\tau^* \not\in  [0,\beta)$. Then there exists an interval
$E=(0,  \varepsilon)$, $0<\varepsilon<h$, such that $ f(s + \tau)- f(\tau) > 0$ and
$M=\max\limits_{r \in[\tau - h,\tau]} u(r) \geq u(s + \tau)$
for all $s \in E$. This implies that the function
$d(s) = u(s+\tau) - u(\tau)$ satisfies the equation
$$d'(s) = a d(s) + f(s + \tau) - f(\tau)$$
for all $s \in E$. Using the variation of constants formula and
the equality $d(0) = 0$,
we get
$$0 \geq d(s) \exp (-a s) =
\int _0^s \exp(-a r)(f(r + \tau)- f(\tau))dr > 0,$$
for all $s \in E$. This contradiction proves that actually $\tau^* \in  [0,\beta)$. \qed 
\end{proof}
Note that a partial converse of Lemma \ref{Lr} is also true:  
\begin{lemma}\label{Lr2}
Let $u: [-h, + \infty) \to \R$ be a solution
of (\ref{meq}). If $u(\tau) = \tilde f(\tau) = \max\limits_{s \in[\tau - h,\tau]} u(s)$, 
where $\tau^*= \tau \pmod{T} \in [0,\beta)$, then $u$ strictly decreases on $I_r:= (\tau, \tau +r)$, where $r= \min\{h, \beta-\tau^*\}>0$. In particular,
 $\tau > 0$ is a point of local maximum  for
$u$. 
\end{lemma}
\begin{proof} Indeed, consider the initial value problem $v(\tau)= \tilde f(\tau)$ for the equation  $v'(t)=av(t)+bv(\tau)+f(t)$. 
The difference 
$m(t) = v(t) - v(\tau)$ satisfies the equation
$$m'(t) = a m(t) + f(t) - f(\tau), \quad m(\tau)=0.$$
Thus, by  the variation of constants formula, for all $t \in I_r$, 
$$(v(t) - v(\tau))\exp (-at) =
\int _\tau^t \exp(-a s)(f(s)- f(\tau))ds <0,$$
proving that  $\tau > 0$ is a point of local maximum  for
$v$ and therefore $u(t)=v(t)$ for all $t \in I_r$. The same computation shows that $u(t_1)=v(t_1) > v(t_2)=u(t_2)$ if $t_1, t_2 \in I_r$ and $t_1 < t_2$, and therefore $u$ is strictly decreasing on $I_r$.
\qed 
\end{proof}

The first recurrence  map ${\cal R}$ plays the same role as the Poincar\'e map in the case of  periodic differential equations.
The following evident statement summarizes the relations between the
delay differential equation (\ref{meq}) and the one-dimensional dynamical system defined by (\ref{R}).
\begin{lemma} \label{L5} For a given solution $u$ of
equation (\ref{meq}), the set of all points satisfying
the third property of Lemma~\ref{L1} forms  a strictly increasing unbounded sequence  $\{\tau_j, j \in \N\}$. Furthermore,  $u(\tau_{n+j}) = {\cal R}^j (u(\tau_n))$ for all $j \geq 0$ and $n \geq 1$. 
\end{lemma}

We emphasize that, clearly,  there is a
correspondence between the periodic solutions of (\ref{meq})
and the set of periodic points of ${\cal R}$. In particular, by \cite{17}, 
${\cal R}$ has at least one fixed point.

\subsection{Existence of a compact global attractor}
Next, we show how the stability assumptions (\ref{ga}) imply that the semiflow $F^{\mu}: X \times \R_+ \to X$ defined in Subsection~\ref{sec:prop} possesses a compact global attractor. 
\begin{lemma} \label{Lem7}
Assume that either of the  two conditions in (\ref{ga}) holds, that $h <T$, and that the T-periodic continuous function $f$ has sine-like shape. 
Then, for each $\phi\in C[-h,0]$, there exist  $K$, which does not depend on $\phi$, and  $t_0=t_0(\phi)$ such that 
\begin{equation}
\label{bounds} K \leq u(t,\tau,\phi) \leq  f^*:= \frac{1}{|a+b|} \max\limits_{t\in [0,T]} f(t),
\end{equation}
for all $t \geq t_0$ and $\tau\in[0,T)$, where $u(t)=u(t,\tau, \phi)$  denotes the solution
of (\ref{1}), (\ref{2}). 
\end{lemma}
\begin{proof} In view of Lemmas~\ref{Lr} and \ref{L5}, there is a sequence $\{\tau_j\}_{j\in\N}$ such that
$u(\tau_j) = -f(\tau_j)/(a+b) \leq f^*$ and $\max\limits_{s \in [\tau_j-h,\tau_j+ \varepsilon_j]}u(s) = u(\tau_j)$, for some $\varepsilon_j>0$.   Take now two consecutive points $\tau_k, \tau_{k+1}$ and suppose that 
$$ u(\tau^*): = \max\{u(t): t \in [\tau_k, \tau_{k+1}]\} > f^*. $$ If $\tau^*$ is the leftmost point with such a property, then necessarily $\tau^* \in \{\tau_j, j \in \N\}$, 
a contradiction.  Therefore $u(t,\tau,\phi) \leq  f^*$ for all sufficiently large $t$. 

By the proof of Lemma~\ref{L2}, we know that for every solution $v: [\tau-h,\infty) \to \R$ of (\ref{eq}) there is a   real number $\mu$ such that  $\mu >2h > 0$  and
$\|v_{\mu}\| \leq 0.5 \|v_{2h}\|$, where $\|\phi\|=\max\{|\phi(s)|:s\in[-h,0]\}.$

On the other hand, in view of Lemma \ref{L5},   if the `universal' constant $K$ in (\ref{bounds}) does not exist, then
there are sequences $s_n \in [0,T), \ \sigma_n \in \R, \
\phi _n(s) \equiv \phi_n \in [\min\limits_{t\in \R} \tilde f(t),
\max\limits_{t\in \R} \tilde f(t)]$, such that  the solutions $u(t,s_n, \phi_n):[s_n-h, +\infty) \to \R$
to (\ref{meq})
satisfy $|u(\sigma _n,s_n, \phi_n)|= \max\limits_{r \in [\sigma_n -\mu,
\sigma_n]}|u(r, s_n, \phi_n)| \to \infty$
as $n \to \infty$. Now, the sequence
$$w^{(n)}(s)= \frac{u(\sigma _n-\mu +s,s_n, \phi_n)}{\max\limits_{r \in [\sigma_n -\mu,
\sigma_n]}|u(r, s_n, \phi_n)|}\, ,\, s\in[0,\mu],$$
is relatively compact in $C([0,\mu])$. Moreover, every limit function $w$
satisfies equation (\ref{eq}) and
$1= |w(\mu)|=\|w_{\mu}\| = \max\limits_{r\in[0,\mu]}|w(r)| \geq \|w_{2h}\|$,
a contradiction with the definition of $\mu$.
 \qed
\end{proof}
\begin{corollary} \label{Co1}Assume all the conditions of Lemma \ref{Lem7} hold. 
Then the  dynamical system $F^{\mu}$ is point
dissipative (see \cite{HMO}). In other words, there exists a positive constant $K_*$ 
 such that for each $x \in X$ we can find $t_0(x)$ satisfying that $F^tx \in S^1 \times \{\phi: \|\phi\| < K_*\}$ for all $t \geq t_0(x)$.
\end{corollary}
\begin{thm}  Assume all the conditions of Lemma \ref{Lem7}  hold. 
Then there exists a compact global attractor ${\cal A}(F)$ for $F^{\mu}$. 
\end{thm}
\begin{proof} The existence of the compact global attractor ${\cal A}(F)$
with the mentioned properties is a direct consequence of Corollary \ref{Co1}
and Theorem 5.3 from \cite{HMO1}.\qed
\end{proof}
It was shown in \cite{17, ILT, kr, PT} that { if $a+b\not =0$, then equation (\ref{meq}) has at least one $T$-periodic solution or, equivalently, 
${\cal A}(F)$ always contains  at least one simple closed curve $\Sigma ^1$
trivially covering the base $S^1$. }Moreover,  
${\cal A}(F)= \Sigma ^1$ under some additional assumptions (e.g. if one of the following three conditions is satisfied: 
i) $|b|+a <0$; \ ii) $(|a|+|b|)h < 1, \ b < -a < 0$; \
iii) $0 < - bh < 3/2$ and $ \ a = 0$). 
In  general, ${\cal A}(F)$ 
does not coincide with $\Sigma ^1$: as it was proved in \cite{PT} (see also Section \ref{S31} below), the global 
attractor can have several periodic orbits.  Moreover, in Section \ref{sec:chaos} of the paper
we will show that ${\cal A}(F)$  can even possess an infinite set of periodic 
solutions as well as some solutions with `chaotic' behavior.

\subsection{Continuity of the return map}

Again, we will assume all the conditions of Lemma \ref{Lem7} hold. We begin by considering the initial value problem $u(\tau)=\tilde f(\tau)$ for 
the ordinary differential equation
\begin{equation}
\label{ode}
u'(t) = (a+b)u(t) + f(t). 
\end{equation}
Clearly, there exists some $\delta >0$ such that $f(t)-f(\tau)$ does not change sign on each of the open  intervals $(\tau-\delta, \tau)$,   $(\tau, \delta +\tau)$. 
A straightforward computation shows that the solution $u$ of the mentioned initial value problem satisfies, for all $0<|t-\tau|<\delta$,
$$
\frac{u(t)-u(\tau)}{t-\tau}(f(t)-f(\tau))= \frac{1}{t-\tau} \int_\tau^te^{(a+b)(t-s)}(f(t)-f(\tau))(f(s)-f(\tau))ds >0. 
$$
This relation implies the following result:
\begin{lemma} \label{LL9} If $\tau \in(0,\beta)$ [respectively, $\tau \in (\beta, T)$] then the solution $u$ of  the initial value problem $u(\tau)=\tilde f(\tau)$ for (\ref{ode}) has a strict local maximum  [respectively, strict local minimum]  at $\tau$. Moreover, $\tau$ is the unique critical point of $u$ in some open neighbourhood of $\tau$. 
If $\tau = \beta$, then $u'(t)>0$ for all $t$ in some punctured  neighbourhood of $\beta$. If $\tau = T$, then $u'(t)<0$ for all $t$ in some punctured  neighbourhood of $T$. 
\end{lemma} 
\begin{proof}Suppose, for example, that $\tau = \beta$. Then $u(\tau) =0$ and $u(t)<0$ for $t \in (\tau-\delta, \tau)$. Since $a+b<0$ this implies that 
$u'(t)>0$ for $t \in (\tau-\delta, \tau)$. Similarly, $u(t)>0$ for $t \in (\tau,\delta+ \tau)$.   If we suppose that $u'(s_0)=0$ for some $s_0 \in (\tau,\delta+ \tau)$ then 
$u(t)$  is strictly decreasing on $(\beta, s_0)$ (since $f(t)$ is strictly increasing on the same interval),
a contradiction. The other cases can be established in a similar fashion. 
\qed 
\end{proof}

Set $K= \tilde f([0,\beta])$. The goal of this subsection
is to describe the continuity properties of the map ${\cal R}: K \to K$.
First, we will analyse the trajectory  of the solution $u(s,p)$
on the interval $(q,\nu^*)$  {(see Subsection~\ref{construction} and Figure~\ref{FigSc} for the notation related to the definition of ${\cal R}$). We next state an assumption that will guarantee good continuity properties 
of ${\cal R}$ and admits practical verification:}

\vspace{2mm}

\noindent ({\bf M}) For each $p \in K$ it holds that $u(s,p) < u(\nu^*,p)$ for all $s \in [\nu^*-h,\nu^*)$. 

\vspace{2mm}

Note that the last assumption needs to be checked only for  those $p$ satisfying  ${\cal R}(p) < p$ because of the following simple  result:
\begin{lemma} \label{LL11} Suppose that $p\in K$ and ${\cal R}(p) \geq p$. Then 
\begin{equation}
\label{ineqM}
u(s,p) < u(\nu^*,p)\, ,\ s \in [\nu^*-h,\nu^*).
\end{equation}
\end{lemma}
\begin{proof} 
If ${\cal R}(p) > p$, the result is a consequence of the definition of ${\cal R}$. Assume that ${\cal R}(p) = p$ and (\ref{ineqM})  does not hold. This means that, if $p= \tilde f(q)$, then $\nu^*-q\leq h$. But this is a contradiction, because Lemma~\ref{Lr} ensures that $\nu^*-q\geq T>h$ if ${\cal R}(p) = p$.
 \qed
\end{proof} 
\begin{lemma} \label{LL12}   Assume that  {\rm({\bf M})} holds. Then for each $p=\tilde f(q) \in K$ there exists a maximal  non-empty 
interval $(\mu (q), \nu^*)$ such that 
$u(t,p) > u(s,p), s \in [t-h,t)$ for  all $t \in (\mu(q), \nu^*)$. In particular, $u'(t,p)> 0$ (possibly, except for one point $s = \beta + jT$   where $u(s,p)=0$) and  $u(t,p)$ satisfies  (\ref{ode}) on $(\mu(q), \nu^*)$. 
\end{lemma}
\begin{proof} Indeed, otherwise there exist increasing  sequences $s_n< t_n< \nu^*, \ s_n, t_n \to \nu^*$ such that 
$
u(t_n) \leq u(s_n) = \max\{u(s), s \in [s_n-h,t_n]\}.  
$
However, this contradicts the definition of $\nu^*$. 

As a consequence, $u(t,p)$ satisfies (\ref{ode}) on $(\mu(q), \nu^*)$. By Lemma \ref{LL9},  $u'(s,p)=0$ if and only if $s = \beta + jT$ for some integer $j$
and  $u(s,p)=\tilde f(\beta)=0$.  \qed
\end{proof}
\begin{corollary}\label{C12a} For each $p \in K$, the inequality 
 ${\cal R}(p) < \max \{\tilde f(t), \ t \in \R\}$ holds.
\end{corollary}
\begin{proof} Indeed, if ${\cal R}(p)=  \max \{\tilde f(t), \ t \in \R\} \geq p$ for some $p\in K$, 
 Lemmas~\ref{LL11} and  \ref{LL12} imply that $u(t, p)$ satisfies (\ref{ode}) and 
increases on some left-hand side neighbourhood
of $\nu^*= 0\pmod{T}$. However, this contradicts the last assertion of    Lemma \ref{LL9}. \qed \end{proof}

\begin{corollary}\label{C12}
  With the notations of Lemma~\ref{LL12}, $\mu(\beta)=\beta$. 
\end{corollary}
\begin{proof}
 Lemma \ref{LL9} implies that  the
solution of the initial problem $u(\beta,0) = 0$ 
increases on some right-hand side neighbourhood
of $\beta$. Therefore the graph of $u(t,0)$ increases until
its first intersection at some point $(z, {\cal R}(0)), \ z \in (\beta, \beta + T),$
with the decreasing part of the graph  of the function
$\tilde f:[T,\beta +T) \to (0, +\infty)$. \qed
 \end{proof}

\begin{thm} \label{T15} Assume  that {\rm({\bf M})} holds and  {$\beta<h$}. Let $p_0=\tilde f(q_0)$ be a point of discontinuity
for ${\cal R}$. Then $u(\beta +jT, p_0) = 0$ for some $\beta + jT \in [\mu(q_0), \nu^*(q_0))$  with $j \in \N$.
Furthermore, 
$$
{\cal R}(p_0)= {\cal R}(0), \quad \mbox{$\displaystyle\liminf_{p \to p_0}{\cal R}(p)=0.$} 
$$
\end{thm}
\begin{proof} Indeed, due to the continuous dependence of solutions on the initial values, 
 $u'(t,p)$ converges uniformly to $u'(t,p_0)$ on $[\mu(q_0), \nu^*(q_0)]$
 as $p\to p_0$.  By Lemma \ref{Lr} and Corollary \ref{C12a}, $\nu^*(q_0) \in (j_0T,j_0T+\beta)$ for some $j_0$. Set $\mu_*(q_0)= \max\{j_0T, \mu(q_0)\}$. 
Then Lemma \ref{LL12} and Corollary \ref{C12} assure that for every $\delta >0$ it holds that 
$u'(t,p)>0$ for all  $t \in [\mu_*(q_0)+\delta, \nu^*(q_0)-\delta)$  if $p$ is sufficiently close to $p_0$. 
This implies that $u(t,p)> u(s,p), $ $ \ s \in   [t-h,t)$, for  $t \in (\mu_*(q_0)+\delta, \nu^*(q_0)-\delta)$ and therefore $u(t,p)$ has a local maximum point $\hat \nu(p)$ such that 
$\hat\nu(p) \to \nu^*(p_0)$ as $p \to p_0$. In addition, $u(t,p)$ is  strictly  monotone in some left and in some right neighbourhoods of  $\hat\nu(p)$ and $u(\hat\nu(p),p)=\max\{u(s,p), \ s \in$  $[\hat\nu(p)-h,\hat\nu(p)]\}$. 

Now, since $\beta <h$, by Lemma~\ref{Lr}, $\hat\nu(p)$ is the absolute maximum point of $u(t,p)$ on the interval $(j_0T, \beta +j_0T)$ containing $\nu^*(q_0)$. 
Therefore, if we suppose that ${\cal R}$ has a discontinuity at $p_0$, it should exist a sequence $p_k=\tilde f(q_k) \to p_0=\tilde f(q_0)$ and a positive integer $j_1< j_0$ such that 
 $\nu^*(q_k) \in [j_1T, \beta +j_1T)$ and  $\nu^*(q_k) \to \mu^*$ $\neq\nu^*(q_0)\pmod{T}.$  Since $u(\nu^*(q_k),p_k)=\max\{u(s,p_k), \ s \in$  $[\nu^*(q_k)-h,\nu^*(q_k)]\}$, in view of  the continuous
 dependence of $u(t,p)$ on the initial data, we  conclude that   $$u(\mu^*,p_0)=\max\{u(s,p_0), \ s \in  [\mu^*-h,\mu^*]\}.$$ 
 
Now,  if $\mu^* < \beta +j_1T$ (i.e. $u(\mu^*, p_0) \not= 0$), then by Lemma \ref{Lr}, 
$u(\mu^*,p_0)=\max\{u(s,p_0), \ s \in  [\mu^*-h,\beta+j_1T]\}$ and $[j_1T, \beta +j_1T)\ni \mu^*=\nu^*(q_0)\in [j_0T, \beta +j_0T)$, a contradiction.  

 Therefore $\mu^* = \beta +j_1T$, 
$$\max\{u(s,p_0), \ s \in  [\mu^*-h,\mu^*]\} = u(\mu^*,p_0) = u(\beta +j_1T,p_0)=0,$$
so that ${\cal R}(p_0)= {\cal R}(0)$  and ${\cal R}(p_k) = u(\nu^*(q_k),p_k)  \to u(\mu^*,p_0) = u(\beta +j_1T,p_0)=0$. 
\qed 
\end{proof}

Theorem~\ref{T15} allows us to find sufficient conditions for the continuity of ${\cal R}$ in some subsets of its domain $K$. We address this task in the next three colloraries.

\begin{corollary} If  $h \in (\beta,T)$, ${\cal R}(p_0) \geq p_0,$  and ${\cal R}(p_0) \not= {\cal R}(0)$,
then ${\cal R}$ is continuous at $p_0$.
\end{corollary}
\begin{proof} It is a consequence of   Lemma~\ref{LL11} and Theorem~\ref{T15}. \qed
\end{proof}

\begin{corollary} \label{cor17} Suppose that $h \in (\beta, T)$,  and let  $p = \tilde f(q)\in K.$ If the following inequality holds:
  \begin{equation}
 \int_0^h e^{a s}(f(q + h - s)-f(q))ds \geq 0,
\label{fs}
\end{equation}
then ${\cal R}$ is continuous at   $p$,   ${\cal R}(p) \geq p,$ and  $\nu^*(q)$ $\in (T,T+\beta).$ Moreover, there exists $r \in (q,q+h]$
such that $u(r,p)=p$, $u'(r,p) >0$, and $u(t,p) < p$ for all $t \in (q,r]$. 
\end{corollary}
\begin{proof} Since ${\cal R}(p)>0$ for all $p \in K$, it suffices to take $p>0$.  Consider the solution $u(t,p)$ on the
interval $I_q=(q,q + h]$. If $u(t,p) \leq p$ for all $t \in I_q$,
and (\ref{fs}) holds,	
we obtain 
\begin{align*}u(q+h,p) &= e^{a h}p    + \int_q^{q + h}
e^{a (h+ q - s)}[b p + f(s)]ds =\\
&=e^{a h}p   + p\frac{b}{a}(e^{ah}-1)+ \int_0^h
e^{as}f(q+h-s)ds \geq p.
\end{align*}

This shows that $u(r,p)=p$ at some leftmost point $r\in (q,q+h]$ where $u'(r,p)\geq 0$. 
 If $u'(r,p)=0$, then 
$p= u(r,p)= \tilde f(r)$, so that $r \in (\beta, T)$ and therefore, repeating the computation in the proof of Lemma \ref{Lr2}, we find that  
$u(t,p) >p$ for  $t \in (\beta, r)$, a contradiction.  
 Thus $u'(r,p)>0$ so that, by the definition of $ \nu^*(q)$, the solution $u(t,p)$ is increasing on the interval $(r, \nu^*(q))$.

Moreover, 
since $h < T$ and $u(t,p) > p \geq 0$ on $(r, \nu^*(q))$, the graph of the solution
does not intersect the set $\{(\beta + jT,0): j \geq 1\}$.  Hence,
by Theorem \ref{T15}, ${\cal R}$ is continuous at $p$ if (\ref{fs}) holds.\qed 
\end{proof} 

\begin{corollary}  \label{cor18}  Suppose that $h \in (\beta, T)$,  $b <0$. Then   ${\cal R}$ is continuous at the point $p$,  if 
\begin{equation}
{\cal R}(p) \not= {\cal R}(0)\quad\mbox{and}\quad
 \int_0^h e^{a s} (f(\nu^*(q) - s)-f(\nu^*(q)))ds <0.
\label{ss}
\end{equation}
\end{corollary}
\begin{proof}
We claim that  $u(s,p) < u(\nu^*,p)$ for all $s \in [\nu^*-h,\nu^*)$. 
Indeed, otherwise $\max\{u(s,p), s \in [t-h,t]\} \geq {\cal R}(p)$ for all $t  \in [\nu^*-h,\nu^*]$ so that 
\begin{align*}{\cal R}(p) &= e^{ah} u(\nu^* -h,p)  + \int_{\nu^* -h}^{\nu^*}
e^{a (\nu^* - s)}[b\max_{w \in [s-h,s]}u(w,p) + f(s)]ds \leq\\
&\leq e^{ah} {\cal R}(p)  + \int_{\nu^* -h}^{\nu^*}
e^{a (\nu^* - s)}[b{\cal R}(p) + f(s)]ds =\\
&=  e^{ah} {\cal R}(p)  + \int_{0}^{h}
e^{a s}[b{\cal R}(p) + f(\nu^* - s)]ds < {\cal R}(p),
\end{align*}
a contradiction.  
Thus condition ({\bf M}) is satisfied and Theorem \ref{T15} implies the continuity of ${\cal R}$ at $p$ once ${\cal R}(p) \not= {\cal R}(0)$.
\qed 
\end{proof}

We illustrate  the application of Corollaries~\ref{cor17} and  \ref{cor18} with an example that we will consider in Section~\ref{sec:chaos}.

\begin{example} 
\label{example32}
{\rm Consider 
 the equation
\begin{equation}
\label{ex2}
u'(t)= 0.32 u(t) - \max_{t-3\pi/2\leq s\leq t}u(s)+1 - \sin t.
\end{equation}
We can easily check that the second condition in  (\ref{ga}) holds. Since $\beta =\pi < h =1.5\pi < T =2\pi$, we can apply Corollaries~\ref{cor17} and \ref{cor18} to conclude that ${\cal R}$ is continuous on the interval $[0,0.9] \subset K=[0,2(a+b)^{-1}]\approx[0,2.94118]$, where ${\cal R}(p)> p$  (by Corollary~\ref{cor17})
and ${\cal R}$ is continuous at each point of the set ${\cal R}^{-1}([1.43\dots, 2.94\dots]\setminus{{\cal R}(0)})$
(by Corollary \ref{cor18}).  The graph of the return map for equation~(\ref{ex2}) is numerically plotted in Figure~\ref{Fig2}. The next result shows that the graph in this figure is a continuous curve at least till its first intersection with the diagonal. Theorem \ref{T24} in Section~\ref{sec:examples}  describes in more detail  the main continuity properties of the return map for (\ref{ex2}). 
}
\end{example}
\begin{corollary} \label{Cor20}
Assume that either of   the stability conditions in 
(\ref{ga}) holds and suppose that $h \in (\beta,T)$.
Then there are $\delta >0$ and  $p^*$ such that ${\cal R} (p^*) = p^*$
and ${\cal R}(p)> p$ for all $p \in [0,p^*)$. Furthermore,  ${\cal R}$ is continuous on $[0, p^* + \delta)$. If, in addition, ({\bf M}) is satisfied and  $[0,c) \subset K$ is the  maximal half-open interval where  ${\cal R}$ is continuous then either $[0,c] = K$ and $\inf{\cal R}(K) >0$ or ${\cal R}(c-)=0, \ {\cal R}(c)= {\cal R}(0)>0$. 
\end{corollary}
\begin{proof}
In view of Lemma \ref{LL9} and Corollary \ref{C12a},  the graph of $u(t,0)$ increases until
its first intersection at some point $(z, {\cal R}(0)), \ z \in (T, \beta + T)$
with the decreasing part of the graph $\Gamma$ of
$\tilde f:(T,\beta +T) \to (0, +\infty)$.  Since $u(t,0)$
has a strict maximum at $z$, it follows that
$$
u(z,0) > u(s,0), \ \mbox{for all} \ s \in [z-h,z+\epsilon]\setminus\{z\}, 
$$
for all small $\epsilon >0$, 
and hence we conclude  that for $p>0$ close to
$0$ the solutions $u(t,p)$ have also strict maxima at some points 
close to $z$.  Therefore the continuous dependence of $u(t,p)$
on the variables $(t,p)$ implies that the point $(\nu^*(q),{\cal R}(p))$ changes
continuously belonging to $\Gamma$ while ${\cal R}(p) \geq p$. By the continuity
of $u(t,p)$, our argument still works when $p$ belongs to some right-hand 
neighbourhood of the  least fixed point $p^*$. 

Finally, if ({\bf M}) holds,  then condition ${\cal R}(p)\geq p$  
can be omitted in the above argumentation and ${\cal R}$ has a continuous graph 
until  the first intersection of its closure with the real axis at some point $c$, where
${\cal R}(c -) = 0, \ {\cal R}(c) = {\cal R}(0) >0$. \qed 
\end{proof}
Corollary \ref{Cor20} provides an alternative proof of the existence of at least one 
$T-$periodic solution for equation (\ref{meq}) with sine-like $T$-periodic continuous function $f(t)$. 
In \cite{17} this result was obtained by using the topological degree method.

\subsection{Differentiability of the return map}
 Hereafter, we again assume that all the conditions of Lemma \ref{Lem7} hold and  $\beta<h$.
It is not difficult to prove the differentiability (possibly, one-side differentiability) of the return map ${\cal R}$ in the case when the graph of 
$u=u(t,p)$ on the interval $(q, \nu^*(q))$  is  
$U$-shaped  in the following sense:
\begin{definition}
We will say that the solution
$u(t,p)$ is $U$-shaped if on
the interval $\Omega_q = (q, \nu^* (q))$ it has only
one critical point, in which it reaches
its minimal value, and if in some left-side
neighborhood of $\nu^*$, $u(t,p)$ satisfies
the ordinary differential equation (\ref{ode}). Set $U(t,p)= \max\{u(s,p), s \in [t-h,t]\}.$
If $u(t,p)$ is $U$-shaped, then the
interval $\Omega_q$ can be represented
as the disjoint union of the subintervals $I_1 = (q, \lambda (q)]$,
$I_2= (\lambda (q), \mu (q)]$ and  $I_3 = (\mu (q), \nu^* (q))$,
where
either $\lambda (q)=q + h,$ or $\lambda(q)= \mu(q)$ and $I_2=\emptyset$, such that $U(t,p) = p$ on $I_1$,
$U(t,p) = u(t-h,p)$ on $I_2$, and
$U(t,p) = u(t,p)$ on $I_3$.
\end{definition}
\vspace{-2mm}
\begin{figure}[htb]
\centering \includegraphics[width=6cm]{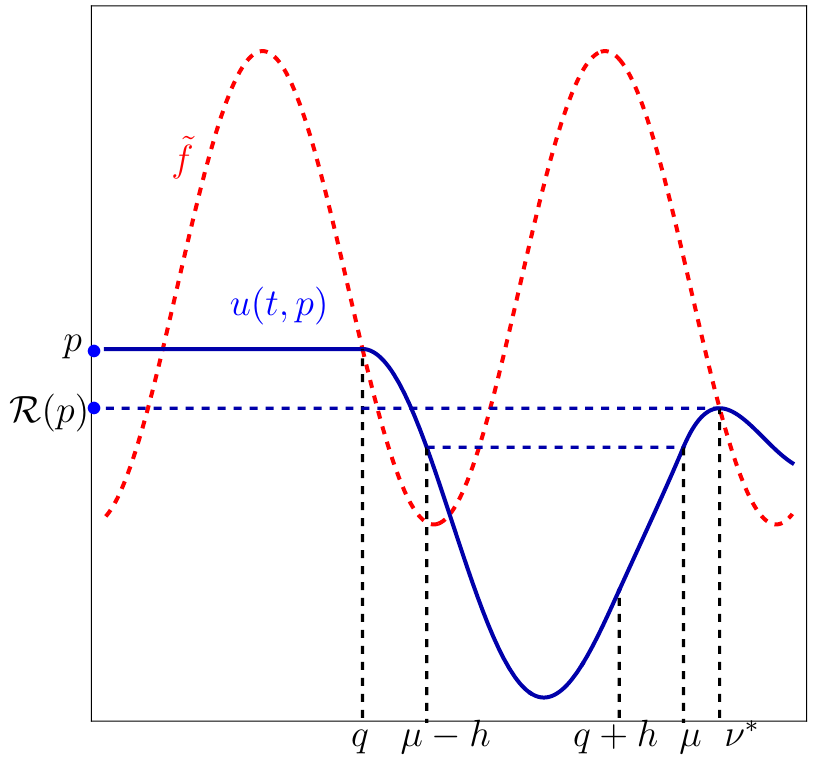}
\caption{Schematic representation of a $U$-shaped solution and its characteristic points.}
\label{FigSc}
\end{figure}

Due to Theorem \ref{T15},
${\cal R}$ is continuous at $p$ if
$u(t,p)$ is $U$-shaped and its graph does not intersect 
the set $\{(\beta + kT,0), k \geq 1 \} \subset \R^2$.

Assuming  $u(t,p)$ is $U$-shaped, we  introduce the following variational  equation along 
$u(t,p)$:

\begin{equation}
w'(t) = 
\begin{cases}  
a w(t) + bw(q), & \text{if $q \leq t \leq  \lambda (q)$,}\\
a w(t) + bw(t-h),& \text{if $ \lambda (q) \leq t < \mu (q)$,}\\
(a +b) w(t), & \text{if $\mu (q) \leq t < \nu^* (q),$}
\end{cases}
\label{ch}
\end{equation}
where $\lambda(q)\in\{q+h,\mu(q)\}$.

Let $v$ denote the fundamental solution of the linear delay-differential equation (\ref{lde}).
Then, if $w(t)$ satisfies the  variational equation (\ref{ch}), we  obtain (see  \cite[Chapter 1, Theorem 6.1]{HD})
hat 
$w(\nu^* (q))= \Delta(q)w(q),$ where
$$\Delta(q)= \left(v(\mu-q)+b\int_{\mu-q-h}^{\mu-q}v(s)ds\right)e^{(a+b)(\nu^*-\mu)}.
$$

To simplify and shorten our proofs, hereafter we assume the following additional assumption that  is fulfilled in the example considered in Subsection~\ref{sec:chaos}:\\

\noindent ({\bf T}) $f$ is  a $C^1$-smooth $T$-periodic function having exactly two critical points on each half-open interval of  length $T$. Moreover, 
$a>0$, $b<0$ and  $h \in (\beta, T)$. \\

Using ({\bf T}), we  can easily establish that $u(t,p)$ has at most one critical point on
 the time interval $(q,T]\cap (q,q+h)$. If $p=0$, this fact follows from Lemma \ref{LL9}.  Next,  Lemma \ref{Lr2} shows that $u(t,p)$ with $p >0$ decreases on some maximal  non-empty  interval $I\supset (q,  \min\{q+h, \beta\})$. In fact, if 
$\hat q> \beta$ is the leftmost point satisfying $\tilde f(\hat q)=p$, then  
$$
u'(t,p) = au(t,p)+ bp+f(t) < (a+b)p +f(q)=0,
$$
for all $t \in (q,  \min\{q+h, \hat q\}]$.

If  $u(t,p)$ has a leftmost critical point $t_m\in (\hat q, q+h)$ then $0\leq u''(t_m)=f'(t_m)$ implying that $\beta < t_m \leq T$ and $0< u''(t_m)$ if $t_m<T$. In particular,  $u(t,p)$  can have at most one critical point 
on $(q,T)$. Now,  suppose that  $u'(T,p)=0$ and  $t_m<T < q+h$. Then $u(t,p) < p$ for all $t \in (q,T]$ so that $v(t)=u'(t,p)$ satisfies $v'(t)=av(t)+f'(t), \ v(T)=0,$ in some small neighbourhood of $T$. Since $f'(t)$ 
is changing its sign at $T$ from positive to negative, $v(t)$ is negative in some small punctured vicinity of $T$. Thus $u(t,p)$ should have an additional local maximum point between $t_m$ and $T$, a contradiction. In this way, $u(t,p)$  can have at most one critical point on $(q,T]\cap (q,q+h)$.

The above reasoning  is useful in proving the following result: 
\begin{lemma}  \label{LL22}Assume that {\rm({\bf T})} and  all the conditions of Lemma \ref{Lem7} hold.  If ({\ref{fs}}) is true then  the graph of
$u=u(t,p)$ is $U$-shaped. 
\end{lemma}
\begin{proof} With the notations of  Corollary~\ref{cor17} and the above comments, it suffices to establish that $t_m$ is  the unique critical point of $u(t,p)$ on the interval $(q,r)$. Indeed, 
if $u(t,p)$ has a different critical point $t_* \in (T,r)\subset (T, \beta+T)$, then $u''(t_*,p) = f'(t_*) <0$ and therefore $t_*$ is the unique critical point of $u(t,p)$
on $(T,r)$ where a local maximum is reached. Thus $u'(r,p) <0$,  which is impossible by   Corollary~\ref{cor17}.
\qed 
\end{proof}
By the same arguments, if $q+h >T$ then $u(t,p)$ can have at most one additional critical point on $(T,q+h]$ where a local maximum is reached. Clearly, this can happen only 
when $u(t,p)< p$ for $t \in (q,q+h]$. Furthermore, suppose that  there exists the leftmost point $r \in (q,q+h]$ such that $u(r,p)= p$. We claim that then the  inequality ({\ref{fs}}) is necessarily satisfied. Indeed, otherwise the solution $u_p(t)$ of the initial value problem $u'(t)=au(t)+bp+f(t),$ $u_p(q)=p$ satisfies $u_p(q+h)<p$
(observe that the inequality $u_p(q+h)\geq p$ amounts to ({\ref{fs}})). Thus $u_p(t)$ reaches its absolute maximum on $[r,q+h]$ at some point $t^* \in (T,q+h]$. 
Since $h< T$ this implies that $f(t^*)>f(q)$ and, consequently, $0=u'_p(t^*)=au(t^*)+bp+f(t^*)>ap+bp+f(t^*) = f(t^*)-f(q) >0$, a contradiction. 

Hence, under the assumptions of Lemma \ref{LL22}, $\mu(q) \leq q+h$ if and only if the  inequality ({\ref{fs}}) holds. For simplicity, it is convenient to consider the following assumption: 

\vspace{2mm}

\noindent ({\bf C}) The set of all $q\in [0,\beta)$ satisfying inequality (\ref{fs}) is a nonempty interval  $S= (\beta_1,\beta)$. 

\vspace{2mm}

For example, condition ({\bf C}) holds for equation (\ref{ex2})  introduced in  Example~\ref{example32} with  $\beta_1\approx 0.39289$ (see Example~\ref{example24}).

By the implicit function theorem, if ({\ref{fs}}) and ({\bf C})  hold then  the equation $u(t,p(q))=p(q)$,  where we denote $p=\tilde f(q)=p(q)$, has a unique solution $t=\lambda(q) \in (q,  q+h],$ smoothly depending on $q \in [\beta_1, \beta)$.  Also $\lambda(q)=\mu(q)$ if $q \in [\beta_1, \beta)$ and $\lambda(q) =q+h$, if $q \leq  \beta_1$.

Next, if  $q \in  (\beta_1,\beta)$, then 
\begin{equation}\label{pp}
p= e^{a( \lambda(q)-q)}p+ \int_q^{ \lambda(q)}e^{a( \lambda(q)-s)}(bp+f(s))ds,
\end{equation}
so that 
$$
1=  \left(p(a+b) +f(\lambda(q))\right)\partial_p\lambda(q)+\left(1+ \frac ba\right)e^{a( \lambda(q)-q)}-\frac b a, 
$$
where $\partial_p$ denotes the partial derivation with respect to $p$.  Next, for $q \in (\beta_1, \beta)$,  
$$
{\cal R}(p)= u(\nu^*(q),p) = pe^{(a+b)(\nu^*(q)-\lambda(q))} +\int^{\nu^*(q)}_{\lambda(q)}e^{(a+b)(\nu^*(q)-s)}f(s)ds,
$$
where $\nu^*$ is $C^1$-smooth function of $q$ as the solution of the equation 
$
F(\nu,q)=0, 
$
where $F(\nu, q)= au(\nu, p(q))+bp(q)+f(\nu), \ \partial_\mu F(\nu, p)=au'(\nu, p)+f'(\nu)=f'(\nu) <0$. 
A straightforward computation shows  that 
\begin{align*}
{\cal R}'(p)=  &u'(\nu^*(q),p)\partial_p\nu^*(q)+ e^{(a+b)(\nu^*(q)-\lambda(q))} -\\
\noalign{\medskip}
&p(a+b)e^{(a+b)(\nu^*(q)-\lambda(q))}\partial_p\lambda(q)  -
e^{(a+b)(\nu^*(q)-\lambda(q))}f(\lambda(q))\partial_p\lambda(q) =\\
\noalign{\medskip}
& e^{(a+b)(\nu^*(q)-\lambda(q))} -e^{(a+b)(\nu^*(q)-\lambda(q))}\left(p(a+b) +f(\lambda(q))\right)\partial_p\lambda(q)=\\
\noalign{\medskip}
&e^{(a+b)(\nu^*(q)-\lambda(q))} \left(\left(1+ \frac ba\right)e^{a( \lambda(q)-q)}-\frac b a \right)=\Delta (q(p)).
\end{align*}
As a  consequence, we have the following result: 
\begin{thm} \label{T23} Assume {\rm({\bf T})} and {\rm({\bf C})} hold, and let $q \in (\beta_1, \beta)$. Then $\mu(q)=\lambda(q)$  and 
\begin{equation}\label{fr}
{\cal R}'(p)= e^{(a+b)(\nu^*(q)-\mu(q))} \left(\left(1+ \frac ba\right)e^{a( \mu(q)-q)}-\frac b a \right)=\Delta (q(p)).
\end{equation}
Therefore, $b > a/(e^{-a h}-1)$ implies that ${\cal R}'(p)>0$ for all $p\in (0, \tilde f(\beta_1))\subset (0,p^*)$. 
On the other hand, if 
$b < a/(e^{-a h}-1)$ and  
the only root of  equation
\begin{equation}
f(\tau) + b \int_0^{\frac{1}{a}\ln\frac{b}{a+b}}e^{-a u}
f(u +\tau)du = 0 
\label{m1}
\end{equation}
is $\tau =q_0\in (\beta_1,\beta)$, then 
 ${\cal R}'(p)>0$  for $p \in (0, \tilde f(q_0))$ and 
 ${\cal R}'(p)<0$ for  $p \in (\tilde f(q_0), \tilde f(\beta_1))$. 
 \end{thm}
\begin{proof} Since $\lambda(q) - q < h$  for all  $q \in (\beta_1,\beta)$, it follows that  $\Delta (q) > 0$ for all  $q \in (\beta_1,\beta)$ if $ h \leq  (1/a)\ln(b/(b+a)).$

  On the other hand, if $ h >   (1/a)\ln(b/(b+a))$, then by the intermediate value theorem, there is $q_0 \in (\beta_1,\beta)$
such that 
$$\lambda(q_0) - q_0 = \frac{1}{a}\ln\left(\frac{b}{b+a}\right)$$
and therefore  ${\cal R}'(p(q_0))=0$. Furthermore, it follows from (\ref{pp}) that $q_0$ satisfies (\ref{m1}),
which  proves the uniqueness of $q_0$ (under our assumptions). Hence, 
$$\lambda^*(q) - q <  \frac{1}{a}\ln\left(\frac{b}{b+a}\right),  \ \mbox{for} \ \,q \in (q_0,\beta); \quad \lambda(q) - q >  \frac{1}{a}\ln\left(\frac{b}{b+a}\right),   \ \mbox{for} \  q \in (\beta_1,q_0),$$ 
which finalises the proof. \qed
\end{proof}
\begin{example}
\label{example24} 
For the equation (\ref{ex2}) considered in  Example~\ref{example32}, we get $\beta =0.5\pi$,  $b<a/(e^{-ah}-1)$, and we numerically find $\beta_1\approx 0.39289$, 
$\tilde f(\beta_1)\approx0.90754$, $q_0\approx1.18459422$,  $\tilde f(q_0)\approx0.10831425.$ Thus, the return map ${\cal R}$ is $C^1$-smooth on the interval 
$[0,0.9]$, where it has a unique critical point $p_c=  \tilde f(q_0)$.  
Moreover, ${\cal R}$ reaches its absolute maximum at $p_c$. See Figure~\ref{Fig2}. 
\end{example}
It is quite remarkable  that the expression for ${\cal R}'(p)$ in (\ref{fr}) does not depend on the derivatives $\partial_q\nu^*(q)$ and $\partial_q\mu(q)$. As one can see in the proof of our next result,  it is due to the following three circumstances: a) that $u'(s,p)=0$ for all $s \in [q-h,q]$ (this eliminates the dependence on  $\partial_q\mu(q)$); b) that $u'(\nu^*(q),p)=0$ (this eliminates the dependence on  $\partial_q\nu^*(q)$); and c) that the graph of $u(t,p)$ is $U-$shaped.  

The next result can be viewed as a natural extension of Theorem \ref{T23} for $q \leq \beta_1$. 
\begin{thm} 
\label{T25} 
Suppose that assumptions {\rm({\bf T}) }and  {\rm({\bf C})} are satisfied,  
equation (\ref{m1}) has a unique  root $\tau \in (\beta_1,\beta)$, and there is $\alpha\in(0,\beta)$ such that the solutions $u(t,p(q))$ of equation (\ref{meq}) 
are  $U$-shaped for all $q \in (\alpha, \beta]$. If  $\Delta(q)<0$ for $q \in (\alpha, \beta_1]$,  then there is an increasing sequence (either finite or infinite) of real numbers $p_i$, with 
$0 < \tilde f(\tau)< p_1 < \dots <p_j < \dots < \tilde f(\alpha),$ 
such that  ${\cal R}$ is differentiable on the intervals 
$D_1= [0,p_1), \dots, D_j= [p_{j-1},p_j), \dots$,
strictly increasing on the interval
$(0, \tilde f(\beta_1))$,
and strictly decreasing on the interval $(\tilde f(\beta_1),p_1 )$
and on every $D_j, \ j > 1.$ Moreover, ${\cal R}(p)$ is right continuous at $p_i$ and ${\cal R}(p_j -) = 0, \ {\cal R}(0) = {\cal R}(p_j )$. 
Finally,
${\cal R}'$ is continuous on every $D_j$ and  ${\cal R}'(p) = \Delta (q(p))$,  ${\cal R}'(p_j-) < {\cal R}'(p_j+)$. 
\end{thm}
\begin{proof} 
By Theorem \ref{T15}, ${\cal R}$ is continuous at $p$ if
$u(t,p)$ is $U$-shaped and if the graph of $u(t,p)$
does not intersect the set $\{(\beta+kT,0): k \in \N,\, k\geq 1\}$
on the interval $(q,\nu^*(q))$.
Suppose that ${\cal R}$ is continuous at a point $p_0=p(q_0)$, with $q_0 < \beta_1.$  We claim that ${\cal R}'(p_0) = \Delta (q_0)$. 
Indeed, since $q_0 < \beta_1$, we have that $\lambda(q_0) = q_0+h < \mu(q_0)$ and therefore, for all $p$ close to $p_0$, it holds 
\begin{align*}
{\cal R}(p) &= u(\nu^*(q),p) = u(\mu(q),p)e^{(a+b)(\nu^*(q)-\mu(q))} +\int^{\nu^*(q)}_{\mu(q)}e^{(a+b)(\nu^*(q)-s)}f(s)ds,\\
\noalign{\medskip}
{\cal R}'(p) &= e^{(a+b)(\nu^*(q)-\mu(q))}\left( \partial_p(u(\mu(q),p)) - ((a+b)u(\mu(q),p)+f(\mu(q))\partial_p\mu(q)\right).
\end{align*}
Now, we have to calculate the partial derivative $\partial_p(u(\mu(q),p))$. A key observation here is that, since $u(t,p)$ is $U$-shaped, it satisfies the following 
delay differential equation on $[q,\mu(q)]$: 
$$
u'(t)= au(t)+bu(t-h)+f(t), \ t \in [q,\mu(q)], \quad u(s)=p, \ s \in [q - h,q]. 
$$
Thus, using the above mentioned fundamental solution $v$, from \cite[Section 1.6]{HD} we obtain that 
$$
u(\mu(q),q)= v(\mu(q)-q)p+ bp\int_{q-h}^qv(\mu(q)-s-h)ds+\int_q^{\mu(q)}v(\mu(q)-s)f(s)ds. 
$$
As a consequence, since $u(\mu(q),p)=u(\mu(q)-h,p)$, we find that
\begin{align*}
\partial_pu(\mu(q),q)&= u'(\mu(q),q) \partial_p\mu(q)+v(\mu(q)-q)+ b\int_{q-h}^qv(\mu(q)-s-h)ds+
\\
&\left(-v'(\mu(q)-q)p-v(\mu(q)-q)f(q) +bpv(\mu(q)-q-h)-bpv(\mu(q)-q)\right)\partial_p q  =
\\
\noalign{\medskip}
& u'(\mu(q),q) \partial_p\mu(q)+v(\mu(q)-q)+
 b\int_{q-h}^qv(\mu(q)-s-h)ds= 
\\
&\left[(a+b)u(\mu(q),q)+f(\mu(q))\right] \partial_p\mu(q)+v(\mu(q)-q)+
 b\int_{q-h}^qv(\mu(q)-s-h)ds. 
\end{align*}
In this way, 
\begin{equation} \label{RDE}
{\cal R}'(p) = e^{(a+b)(\nu^*(q)-\mu(q))}\left[ v(\mu(q)-q)+
 b\int_{q-h}^qv(\mu(q)-s-h)ds\right] = \Delta(q). 
\end{equation}
Next, integrating  equation (\ref{ch}) we find that  $\Delta (q)$
is a combination of some  elementary functions depending
on $a$, $b$, $h$, $\lambda (q)$, $\mu (q)$ and $\nu^*(q)$. 
The continuous dependence of $\lambda (q)$, $\mu (q)$ and $\nu^*(q)$ on $q$
belonging to some small neighbourhood ${\cal O}$ of $q_0$ implies the 
continuity of ${\cal R}'(p) = \Delta (q)$ in ${\cal O}$.

Observe also that the sign of ${\cal R}'(p)$ is completely defined by the factor given in brackets in  (\ref{RDE}). 
In view of  the $U$-shaped form of $u(t,p)$, the function $\mu(q)$ is $C^1$-smooth so that the aforementioned factor depends continuously on $p$. Differently,
$\nu^*(q)$ is discontinuous at the preimages of the discontinuity points $p_j=\tilde f(q_j)$ of ${\cal R}$.  Assuming that there exist ${\cal R}'(p_j+)$ and 
${\cal R}'(p_j-)$, we find immediately that 
$$
{\cal R}'(p_j+) = {\cal R}'(p_j-)  e^{(a+b)(\nu^*(\beta)-\beta)} < {\cal R}'(p_j-). 
$$

By Corollary \ref{Cor20}, either ${\cal R}$ is continuous on $K$ or there exists a leftmost discontinuity point $p_1$. 
In the first case,  ${\cal R}$ has a unique critical point $\tilde f(\tau)$ on $K$ and $\nu^*(q) < T+\beta$ for all $q$. In the second case, ${\cal R}$ is continuous and strictly decreasing on $[\tilde f(\tau),p_1)$, with  ${\cal R}(p_1-) = 0, \ {\cal R}(0) = {\cal R}(p_1)$.  

Next,  we claim that $\nu^*(q) > \beta+T$ for $p>p_1$. Indeed, if $\nu^*(\hat q) < \beta+T$ for some $\hat p = \tilde f(\hat q) >p_1$, then the negativity of 
${\cal R}'(p)$ yields  $\nu^*(q) < \beta+T$ for all $q \in (\hat q, \beta)$, a contradiction. Therefore, considering the $U$-shaped form of $u(t,p)$ and 
the inequality ${\cal R}'(p)<0$,  we conclude that the graph of $u(t,p)$ does not contain the point $(\beta+T,0)$ for $p>p_1$.  This allows us to repeat the argumentation 
of Corollary \ref{Cor20} for the case when $\nu^*(q) \in (2T, \beta+2T)$. In particular, we obtain that  ${\cal R}$ is continuous and strictly decreasing on  some maximal open right neighbourhood  ${\cal O}_1$ of $p_1$ and that, if  $p_2:= \sup {\cal O}_1$
is an interior point of $K$, then ${\cal R}(p_2-) = 0, \ {\cal R}(0) = {\cal R}(p_2)$. 

By applying repeatedly the above procedure, we construct the sequence $\{p_j\}$ with the properties mentioned in the statement of the theorem. 
\qed
\end{proof}

\begin{corollary}  \label{cor26}  Let $D_j$ be the intervals defined in the statement of Theorem~\ref{T25}.
Suppose that ${\cal R}(0) \in D_m$ for some $m\geq 1$. Then  equation (\ref{meq}) has $m$ sine-like periodic solutions $p_j(t)$
with  minimal periods $jT$ and such that $\zeta_j:=\max_\R p_j(t) < \max_\R p_k(t)$ for each pair of indexes $j <k$. 
\end{corollary}
\begin{proof} If ${\cal R}(0) \in D_m$ then ${\cal R}$ has exactly
$m$ fixed points $\zeta_i \in D_i, \ i = 1,\dots,m$. 
\qed 
\end{proof}
\begin{remark}
\label{R27} Let $v$ denote the fundamental solution of  (\ref{lde}), and consider the function 
$$
V(t) = v(t)+b\int_{t-h}^{t}v(s)ds, \quad t \geq 0. 
$$
Note that $V(0)=1$. Suppose that $V(t) >0$ for $t \in [0,\alpha_*)$ and $V(t) <0$ for $t \in (\alpha_*, \beta_*)$. 
Assume that all conditions of Theorem \ref{T23} hold, and let each $u(t,p)$ be $U$-shaped.  Then 
the inequalities $\alpha_*< \mu (q)- q <  \beta_*$ clearly guarantee  that $\Delta (q) < 0$.  

As an application, consider the equation (\ref{ex2}) defined in  Example~\ref{example32}, for which $\alpha_* \approx1.2,$ $\beta_*\approx12.11$.  Since $\mu(q)-q > 1.5\pi> \alpha_*$
for $q < q_0$, and  $\mu (q)- q< 12.11$ if $\mu(q) \leq 1.5\pi,\ q \geq - 0.5\pi$, we can conclude that ${\cal R}'(p) <0$ for all $p\in (\tilde f(q_0),p_1)$  (in the Appendix, we will prove that the corresponding solutions $u(t,p)$ are $U$-shaped).

\end{remark}

\section{Two examples}\label{sec:examples}
In this section, we give two applications of our results.

\subsection{Equation with  multiple attracting solutions}\label{S31}
The equation 
\begin{equation}
u'(t)=-\max \limits_{s\in [t-3\pi/2,t]}u(s) +f(t),
\label{ex1}
\end{equation}
with $f(t) = - \sin t + \max_{t-3\pi/2\leq \tau\leq t} \cos \tau$ was studied in  \cite{PT}.
Function $u_1(t) = \cos t$ is an evident solution of (\ref{ex1}) and 
the existence of another $8\pi$-periodic solution $u_2$ was established in the cited work. However,   the  full description of the dynamics of  (\ref{ex1}) was not provided in \cite{PT}. This can be easily done by analysing the return map ${\cal R}$ for (\ref{ex1}), whose graph  is presented in Figure~\ref{Fig1}.
We see that, in fact, the minimal period
of $u_2$ is $4\pi$.
Moreover, $u_1$, $u_2$ and $u_3(t)=u_2(t+ 2\pi)$
exhaust the set of all periodic solutions to (\ref{ex1}),
and $u_2$, $u_3$ attract all solutions
to (\ref{ex1}) (clearly, excepting $u_1$).
We find that  ${\cal R}'(1) = (1-7\pi/4+\pi^2/32)\exp(-\pi/4) \approx- 1.91$, 
which coincides with the unique  non-zero characteristic multiplier
determined by the variational equation along $u_1(t) = \cos t$
(see \cite[Theorem 1.2]{PT} for more details).

%\vspace{-25mm}

\begin{figure}[htb]
% \hspace{-10mm}
{\includegraphics[width=6.5cm]{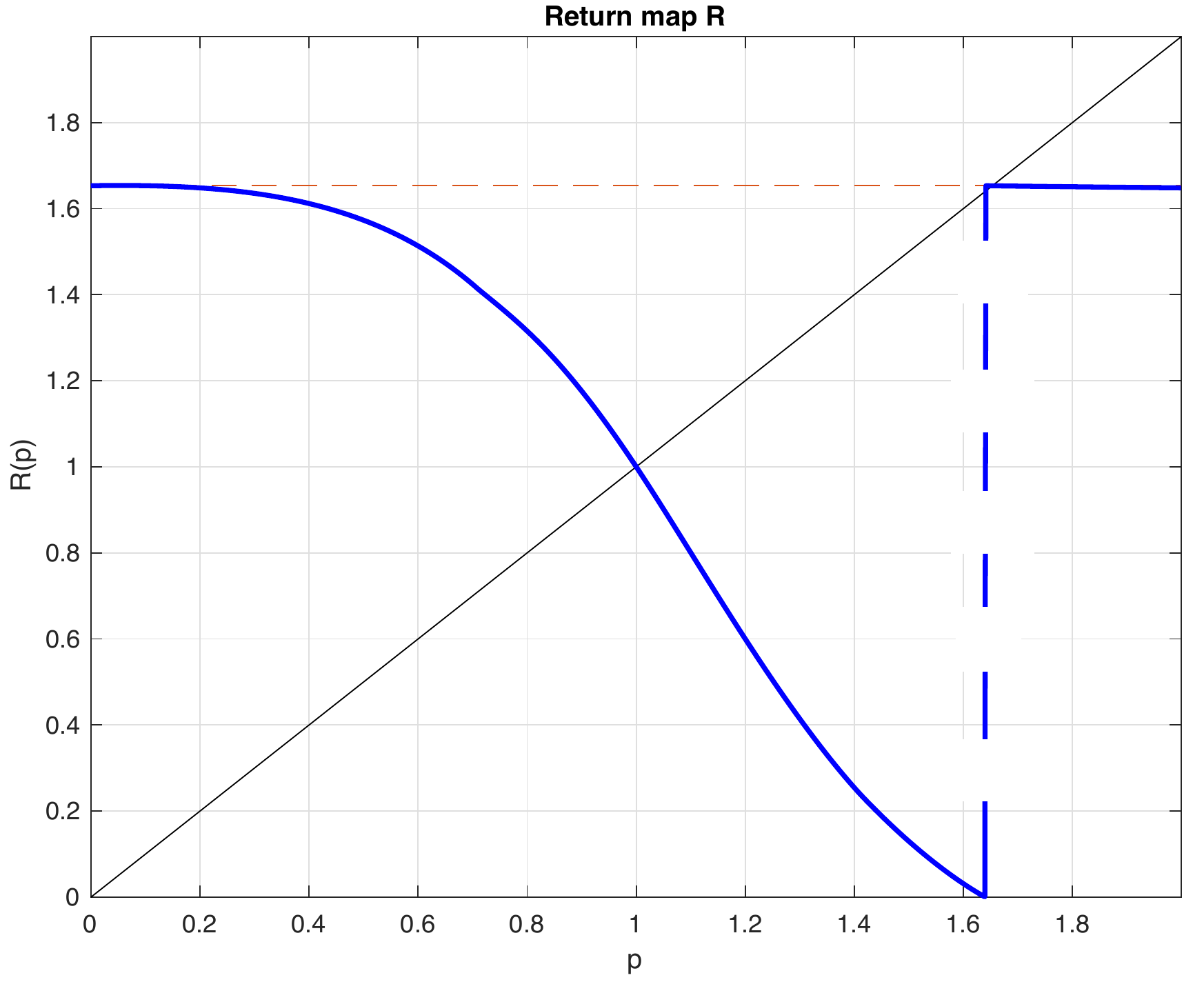}\hspace*{0.5cm}\includegraphics[width=6.5cm]{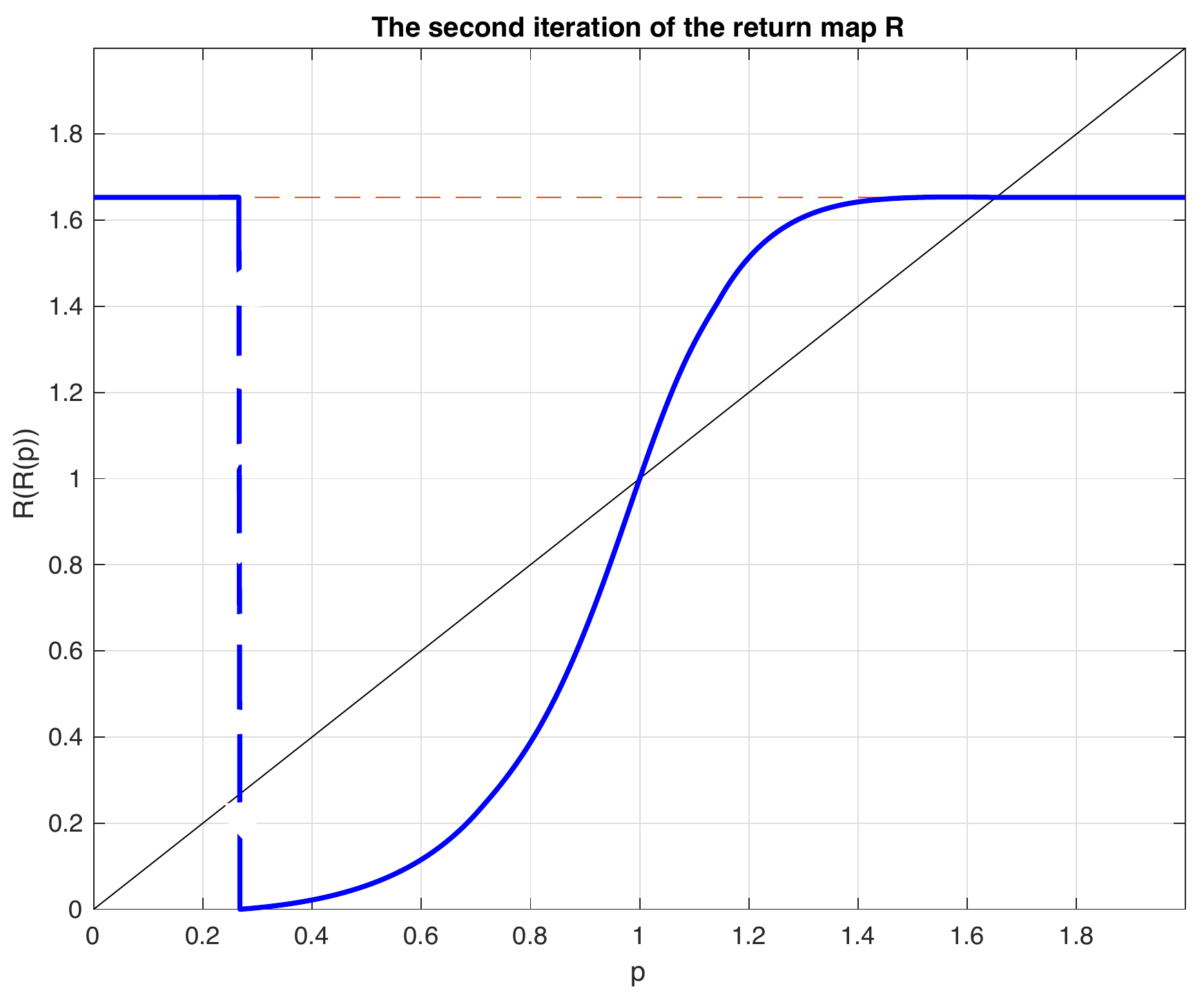}}
%\vspace{-30mm}
\caption{  Return map ${\cal R}$ for equation (\ref{ex1}) (on the left) and its second iteration ${\cal R}^2$ (on the right). The dashed horizontal line corresponds to ${\cal R}(0)$. Note that ${\cal R}^2$ and ${\cal R}$ share one point of discontinuity (not appreciable on the right frame).} 
\label{Fig1}
\end{figure}

\begin{figure}[htb]
\centering \includegraphics[width=8cm]{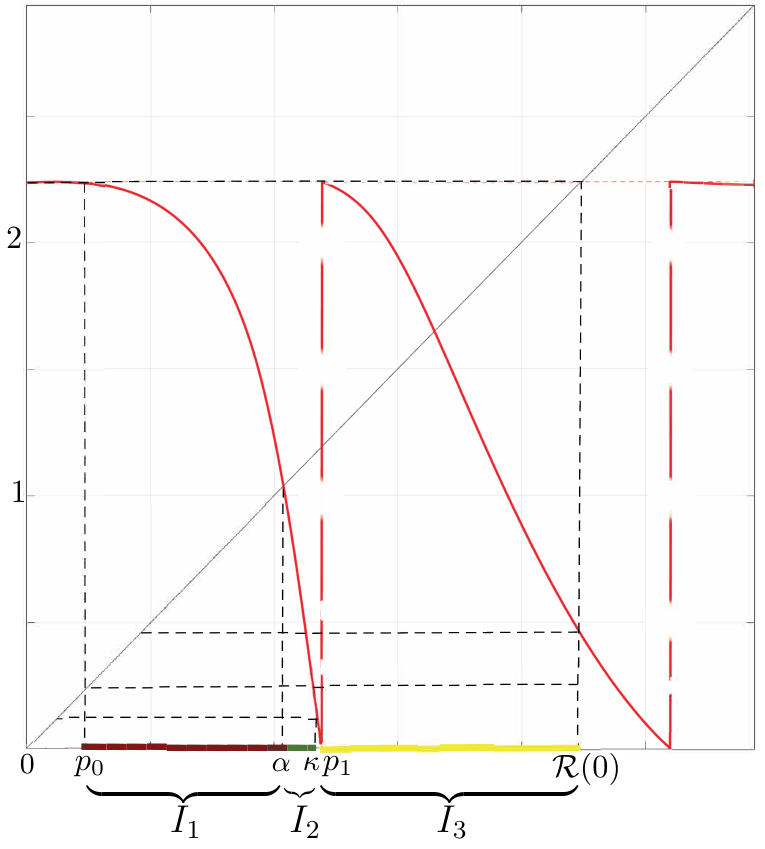}
\caption{Graph of the return map ${\cal R}$ for equation (\ref{ex2}) (discontinuous red solid curve). The upper dashed horizontal line corresponds to ${\cal R}(0)$. 
The closed subintervals $I_1$ (brown), $I_2$ (green), $I_3$ (yellow) of $K$ have pairwise disjoint interiors and satisfy the relations $I_2\cup I_3 \subset {\cal R}(I_1)$, $I_1 \subset {\cal R}(I_2) $ , $I_2\cup I_3 \subset {\cal R}(I_3)$.}
\label{Fig2}
\end{figure}

\subsection{Chaotic behavior in the Magomedov equation}
\label{sec:chaos}
The forcing term $f(t)$ in (\ref{ex1}) is 
close to the  function $g(t) = 1 - \sin t$. In fact, the replacement of  
$f(t)$ with $1 - \sin t$ in (\ref{ex1}) produces dynamically  insignificant changes in the return
map so that the modified system has  the same
simple dynamics. However, by adding the linear  term
$0.32 u(t)$ to (\ref{ex1}),  the behavior of the solutions changes dramatically. \footnote {The specific choice of the parameters $a=0.32$ and $b=-1$ 
is mostly motivated by some advantages in the graphical representation of the solutions and in establishing the continuity properties of $\cal R$  in the Appendix. 
Note also that, for some $a, b$, the map $\cal R$ can have an attracting cycle with a large basin of attraction. This possibility is excluded by the above choice of parameters (the rightmost continuous branch of the graph of ${\cal R}$  does not intersect the diagonal, compare with the left part of Figure~\ref{Fig1}).}
 Indeed, as we show below,  equation~(\ref{ex2}) introduced in Example~\ref{example32}
exhibits chaotic behaviour.

In the Appendix, we prove the following result (see also Figure~\ref{Fig2}, which represents the 
return map ${\cal R}$ for  (\ref{ex2})).

%\newpage

\begin{thm} \label{T24} The return map ${\cal R}:K\to K$ for  (\ref{ex2}) has exactly two points of discontinuity $p_1\approx1.11$ and $p_2\approx2.61$
on the interval $[0,p_2]\supset {\cal R}(K)$, where 
$${\cal R}(p_1)={\cal R}(p_2)= {\cal R}(0)\approx2.23,  \quad {\cal R}(p_1-)= {\cal R}(p_2-)=0.$$
Furthermore, ${\cal R}$ is differentiable on $[0,1.316]\setminus\{p_1\}$ and  has a unique critical point $p_c\approx0.108$ on this interval, where it reaches its absolute maximum. Finally, ${\cal R}'(p_1-)<{\cal R}'(p_1+)$, ${\cal R}(p) >p$ for all 
$p \in [0,0.9]$ and ${\cal R}({\cal R}(0)) < 0.9$. 
\end{thm}
This theorem implies the existence of a leftmost fixed point $\alpha \in (0.9,p_1)$ for ${\cal R}$. Let  $p_0 \in (0,p_1)$ be defined by ${\cal R}(p_0)={\cal R}(0)$ and let $\kappa \in (\alpha, p_1)$ be sufficiently close to 
$p_1$ to satisfy ${\cal R}(\kappa) < p_0$.   Consider the following closed subintervals of $K$ with  pairwise disjoint interiors
$$
I_1=[p_0,\alpha], \quad I_2 = [\alpha, \kappa], \quad I_3 = [p_1, {\cal R}(0)]. 
$$
These intervals are shown in Figure~\ref{Fig2}.  Clearly, the return map is continuous on each of these intervals and 
$$I_2\cup I_3 \subset {\cal R}(I_1), \quad I_1 \subset {\cal R}(I_2), \quad I_2\cup I_3 \subset {\cal R}(I_3).$$
Writing  the inclusion $I_1 \subset {\cal R}(I_2)$ in the form $I_2\to I_1$, and similarly the others, we obtain the following directed Markov graph associated with the collection 
$I_1, I_2,I_3$: 

\vspace{2mm}

\scalebox{0.7}
{
\tikzset{every picture/.style={line width=0.75pt}} %set default line width to 0.75pt        
\begin{tikzpicture}[x=0.75pt,y=0.75pt,yscale=-1,xscale=1]
\path (-100,300); %set diagram left start at 0, and has height of 300

%Straight Lines [id:da33897711030177713] 
\draw    (157,135) -- (278.13,88.71) ;
\draw [shift={(280,88)}, rotate = 519.0899999999999] [color={rgb, 255:red, 0; green, 0; blue, 0 }  ][line width=0.75]    (10.93,-3.29) .. controls (6.95,-1.4) and (3.31,-0.3) .. (0,0) .. controls (3.31,0.3) and (6.95,1.4) .. (10.93,3.29)   ;
%Straight Lines [id:da16053701016939748] 
\draw    (276,96) -- (165.87,138.28) ;
\draw [shift={(164,139)}, rotate = 339] [color={rgb, 255:red, 0; green, 0; blue, 0 }  ][line width=0.75]    (10.93,-3.29) .. controls (6.95,-1.4) and (3.31,-0.3) .. (0,0) .. controls (3.31,0.3) and (6.95,1.4) .. (10.93,3.29)   ;
%Straight Lines [id:da14106755742506305] 
\draw    (160,147) -- (261.14,187.26) ;
\draw [shift={(263,188)}, rotate = 201.71] [color={rgb, 255:red, 0; green, 0; blue, 0 }  ][line width=0.75]    (10.93,-3.29) .. controls (6.95,-1.4) and (3.31,-0.3) .. (0,0) .. controls (3.31,0.3) and (6.95,1.4) .. (10.93,3.29)   ;
%Straight Lines [id:da5119981792768235] 
\draw    (275,172) -- (289.58,103.96) ;
\draw [shift={(290,102)}, rotate = 462.09] [color={rgb, 255:red, 0; green, 0; blue, 0 }  ][line width=0.75]    (10.93,-3.29) .. controls (6.95,-1.4) and (3.31,-0.3) .. (0,0) .. controls (3.31,0.3) and (6.95,1.4) .. (10.93,3.29)   ;
%Curve Lines [id:da41245489982140704] 
\draw    (271,197) .. controls (251.1,284.56) and (388.61,226.59) .. (292.47,189.56) ;
\draw [shift={(291,189)}, rotate = 380.49] [color={rgb, 255:red, 0; green, 0; blue, 0 }  ][line width=0.75]    (10.93,-3.29) .. controls (6.95,-1.4) and (3.31,-0.3) .. (0,0) .. controls (3.31,0.3) and (6.95,1.4) .. (10.93,3.29)   ;

% Text Node
\draw (122,130) node [anchor=north west][inner sep=0.75pt]   [align=left] {$ $};
% Text Node
\draw (139,133) node [anchor=north west][inner sep=0.75pt]    {$I_{1}$};
% Text Node
\draw (290,75) node [anchor=north west][inner sep=0.75pt]    {$I_{2}$};
% Text Node
\draw (272,176) node [anchor=north west][inner sep=0.75pt]    {$I_{3}$};
\end{tikzpicture}
}

\vspace{-5mm}

The  adjacency matrix ${\cal A}=\{a_{ij}\}$ of the graph  is defined as follows: $a_{ij}=1$ if and only if there is an edge from vertex $I_i$ to vertex $I_j$; otherwise,  $a_{ij}=0$. Thus:
$$
{\cal A}=   \begin{bmatrix}
0 & 1 & 1\\
1  & 0 & 0 \\
0 & 1  & 1
\end{bmatrix}. $$
Consider the space $\Omega_{\cal A}^+$ of all one-sided paths on the above Markov graph (for example, $\omega=(I_3,I_3,I_2,I_1,I_2,I_1,I_3, \dots)$)
provided with the metrisable  topology of component-wise convergence. It is easy to realise that $\Omega_{\cal A}^+$ is a closed perfect subspace of the product space $\{I_1,I_2,I_3\}^\N$ so that it is a Cantor set. Let $\sigma: \Omega_{\cal A}^+\to \Omega_{\cal A}^+$ denote the one-sided shift defined by $\sigma(\{I_{n_k}\})=\{I_{n_{k+1}}\}$ 
(e.g. $\sigma(I_3,I_3,I_2,I_1,I_2,I_1,I_3, \dots)=(I_3,I_2,I_1,I_2,I_1,I_3, \dots)$). Since all elements of the matrix ${\cal A}^3$ are positive (hence,  the matrix $\cal A$ is transitive \cite[Definition 1.9.6]{KH}),  the dynamical system  $\sigma: \Omega_{\cal A}^+\to \Omega_{\cal A}^+$  is topologically mixing and its periodic points are dense in $\Omega_{\cal A}^+$, see  \cite[Proposition 1.9.9]{KH}. Then, an application of Theorem 15.1.5,  Corollaries 1.9.5, 15.1.6, 15.1.8, and  Proposition~3.2.5  in \cite{KH}
yields the following result (notice that the greatest eigenvalue of  ${\cal A}$ is $\lambda_{{\cal A}}^{max}=(\sqrt{5}+1)/2$).

%\vspace{7mm}

\begin{thm} \label{T29} There is a closed subset $J \subset I_1\cup I_2 \cup I_3$ and a continuous surjection $h:J \to \Omega_{\cal A}^+$ such that ${\cal R}(J)\subset J$ and the following diagram is commutative 

\vspace{2mm}

\scalebox{0.75}
{
\tikzset{every picture/.style={line width=0.75pt}} %set default line width to 0.75pt        

\begin{tikzpicture}[x=0.75pt,y=0.75pt,yscale=-1,xscale=1]
 \path (-50,300); %set diagram left start at 0, and has height of 300

%Straight Lines [id:da0674972766145221] 
\draw    (149,149) -- (290,150.97) ;
\draw [shift={(292,151)}, rotate = 180.8] [color={rgb, 255:red, 0; green, 0; blue, 0 }  ][line width=0.75]    (10.93,-3.29) .. controls (6.95,-1.4) and (3.31,-0.3) .. (0,0) .. controls (3.31,0.3) and (6.95,1.4) .. (10.93,3.29)   ;
%Straight Lines [id:da07024733363624414] 
\draw    (147,230) -- (287,230) ;
\draw [shift={(289,230)}, rotate = 180] [color={rgb, 255:red, 0; green, 0; blue, 0 }  ][line width=0.75]    (10.93,-3.29) .. controls (6.95,-1.4) and (3.31,-0.3) .. (0,0) .. controls (3.31,0.3) and (6.95,1.4) .. (10.93,3.29)   ;
%Straight Lines [id:da9852079213680993] 
\draw    (133,163) -- (132.04,212) ;
\draw [shift={(132,214)}, rotate = 271.12] [color={rgb, 255:red, 0; green, 0; blue, 0 }  ][line width=0.75]    (10.93,-3.29) .. controls (6.95,-1.4) and (3.31,-0.3) .. (0,0) .. controls (3.31,0.3) and (6.95,1.4) .. (10.93,3.29)   ;
%Straight Lines [id:da8364225734522519] 
\draw    (308,166) -- (308,217) ;
\draw [shift={(308,219)}, rotate = 270] [color={rgb, 255:red, 0; green, 0; blue, 0 }  ][line width=0.75]    (10.93,-3.29) .. controls (6.95,-1.4) and (3.31,-0.3) .. (0,0) .. controls (3.31,0.3) and (6.95,1.4) .. (10.93,3.29)   ;

% Text Node
\draw (122,137) node [anchor=north west][inner sep=0.75pt]    {$J$};
% Text Node
\draw (122,134) node [anchor=north west][inner sep=0.75pt]    {$$};
% Text Node
\draw (303,139) node [anchor=north west][inner sep=0.75pt]    {$J$};
% Text Node
\draw (122,223) node [anchor=north west][inner sep=0.75pt]    {$\Omega_{\cal A}^+$};
% Text Node
\draw (303,223) node [anchor=north west][inner sep=0.75pt]    {$\Omega_{\cal A}^+$};
% Text Node
\draw (206,130) node [anchor=north west][inner sep=0.75pt]    {$\cal R$};
% Text Node
\draw (210,215) node [anchor=north west][inner sep=0.75pt]    {$\sigma$};
% Text Node
\draw (110,177) node [anchor=north west][inner sep=0.75pt]    {$h$};
% Text Node
\draw (322,179) node [anchor=north west][inner sep=0.75pt]    {$h$};
\end{tikzpicture}
}
\vspace{-10mm}

\noindent Moreover, to each periodic orbit $\omega \in \Omega_{\cal A}^+$ corresponds at least one periodic point of the same period in $h^{-1}(\omega)$
so that ${\cal R}: J \to J$ has an infinite set of periodic solutions. In fact, 
the number of different $n$-periodic orbits of $\cal R$ is bigger than or equal to  the trace $Tr {\cal A}^n$, and the topological entropy of ${\cal R}: J \to J$ is at least $\log  ((\sqrt{5}+1)/2)>0$. 
\end{thm}
In this way,  equation (\ref{ex2}) has an infinite set of periodic solutions. In particular,  Figure~\ref{Fig2} shows that it has a 
$2\pi$-periodic solution $u_1$ with
$\alpha = \max_{t \in \R}u_1(t) \approx1.037,$
and a  $4\pi$-periodic solution $u_2$ with
$\gamma = \max_{t \in \R}u_2(t) \approx1.65$. The graph
of the second iteration ${\cal R}^2$ restricted
to the interval $[\alpha, \gamma]$ suggests that  ${\cal R}$ has an infinite set 
of unstable periodic solutions. In Figure~\ref{Fig3}, we represent two particular solutions of equation (\ref{ex2}): the  curve $(t,u(t))$ of the  solution $u=u(t,1)$, $t<350$, and the projection of the solution $u=u(t,0)$ on the plane $(u(t), u(t-h))$.

\begin{figure}[htb] 
\centering \includegraphics[width=6.9cm]{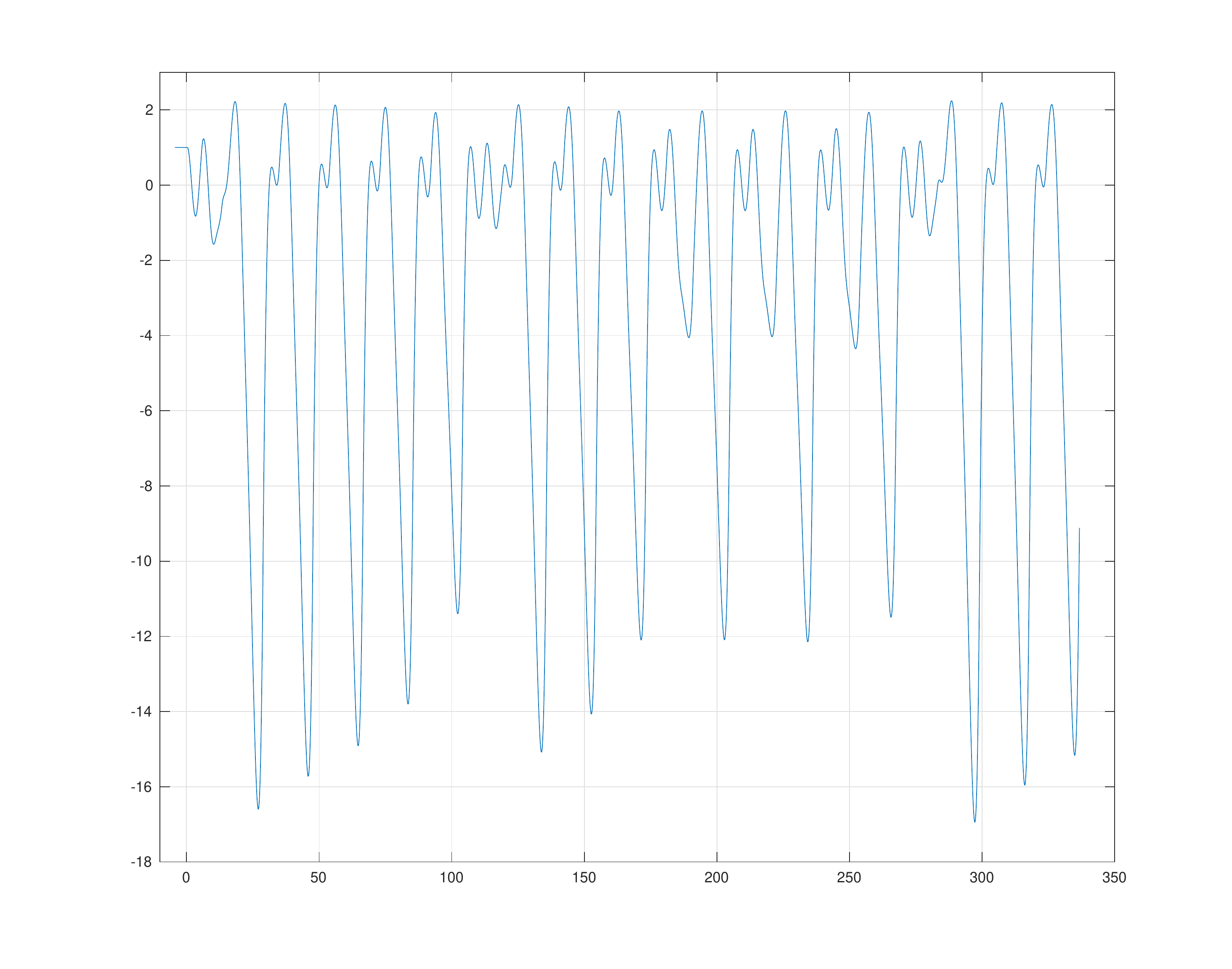} \includegraphics[width=6.4cm]{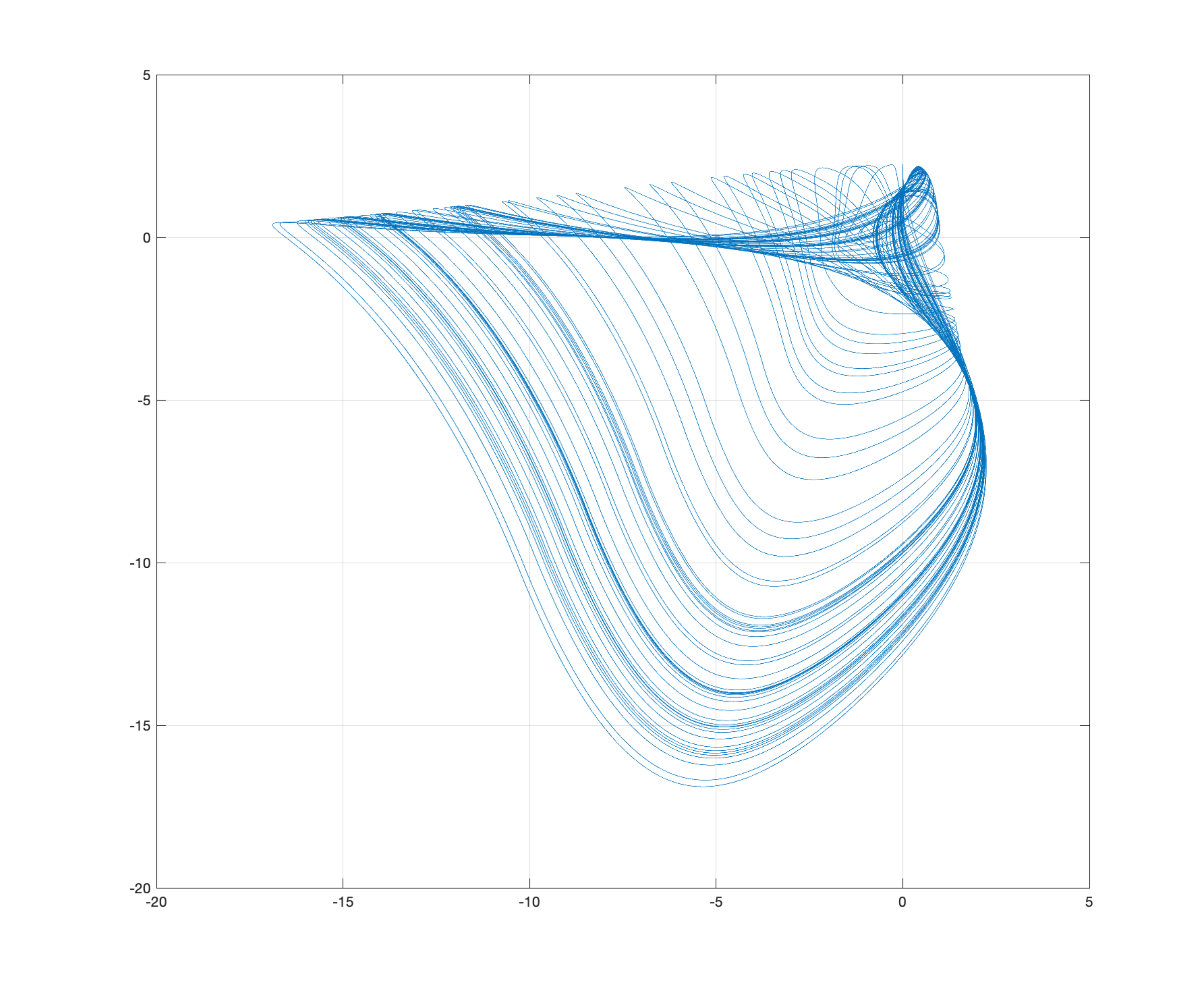}
\caption{On the left, we represent the solution $u=u(t,1)$ of equation (\ref{ex2}); on the right, the projection of  the solution $u=u(t,0)$ of (\ref{ex2}) on the plane $(u(t,0), u(t-h,0))$.} 
\label{Fig3}
\end{figure}

\section*{Appendix}
Here, we present a proof of  Theorem \ref{T24} based on the analysis of the explicit formulae for  the solutions 
of the initial value problem 
\begin{equation}
u'(t)= au(t) - u(t-h)+1 - \sin t, \ u(s)\equiv p= (1-\sin q)/d, \ s \in [q-h,q],
\label{x2}
\end{equation}
where $a = 0.32$, $h= 3\pi/2$, $d =0.68$ and $q \in [-0.5\pi, 0.4]$ (note that, by Example~\ref{example24}, $\beta_1 \approx 0.39289<0.4$). 
For our purposes, it suffices to integrate (\ref{x2}) on two steps: $[q,q+h]$ and $[q+h,q+2h]$. First, we observe that 
the unique periodic solution $p(t)$  of the ordinary differential equation $u'(t)= au(t) +k \sin (t+\varphi)$ has the form
$$
u(t)=-\frac{k}{\sqrt{a^2+1}}\sin(t+\varphi +\theta_0),\ \mbox{where} \ \theta_0: = \arcsin\frac{1}{\sqrt{a^2+1}}.
$$
Thus, integrating (\ref{x2}) on $[q,q+h]$, we easily find that 
\begin{equation}\label{u1}
u(t)= C_0\sin(t+\theta_0)+ C_1+ C_2e^{a(t-q)}, \ \mbox{where} \ C_0=\frac{1}{\sqrt{a^2+1}}, \ C_1= \frac{p-1}{a}, 
\end{equation}
$$C_2=p-C_1-C_0\sin(q+\theta_0).  
$$
Hence, solving (\ref{x2}) on $[q+h,q+2h]$ amounts to the integration of the linear inhomogeneous ordinary differential equation 
\begin{equation}
u'(t)= au(t) - (C_0\sin(t-h+\theta_0)+ C_1+ C_2e^{a(t-h-q)})+1 - \sin t, \ 
\label{x3}
\end{equation}
subject to the initial condition 
\begin{equation}\label{x4}
u(q+h)= -C_0\cos(q+\theta_0)+ C_1+ C_2e^{ah}=:C_3. 
\end{equation}
The solution of (\ref{x3}), (\ref{x4}) is given by  
\begin{equation} \label{sec}
u(t)= C^*_0\cos(t+2\theta_0)+C_0\sin(t+\theta_0)+ C^*_1- C_2(t-h-q)e^{a(t-h-q)}+ C^*_2e^{a(t-h-q)}, 
\end{equation}
where
$$
\ C^*_0=\frac{1}{a^2+1}, \ C^*_1= \frac{C_1-1}{a}, \ C_2^*= C_3- C^*_0\sin(q+2\theta_0)+C_0\cos(q+\theta_0)- C^*_1.$$

This implies that the first derivative $u'(t)$ is an analytic function of the variables $q \in [-0.5\pi, 0.4]$ and $t \in [q+h,q+2h]$:
$$
u'(t)= -C^*_0\sin(t+2\theta_0)+C_0\cos(t+\theta_0)+\left[-C_2a(t-h-q) + (C^*_2a- C_2)\right]e^{a(t-h-q)}.$$ 
Note also that
$$
u(q+2h)-u(q+h)= -C^*_0\cos(q+2\theta_0)-C_0\sin(q+\theta_0)+ C^*_1- C_2he^{ah}+ C^*_2e^{ah}-C_3.
$$ 
\begin{lemma} For all  $q \in [0.105, 0.4]$, $t \in [q+h,q+2h]$, it holds that $u'(t)>0$. Furthermore, $u(q+2h)-u(q+h)>0$
for all  $q \in [-0.12, 0.4]$. 
\end{lemma}
\begin{proof} It is convenient to introduce the new variable $s=t-q-h\in [0,h]=[0,1.5\pi]$. Then we have to evaluate the elementary function 
$$
\Psi(s,q)= C_0\sin(s+q+\theta_0)+ C^*_0\cos(s+q+2\theta_0)+\left[-C_2as + (C^*_2a- C_2)\right]e^{as}
$$
on the rectangle $\Pi= [0,1.5\pi]\times[0.105,0.4]$. Since $\min\{\Psi(s,q), \ (s,q) \in \Pi\} = 0.0086\dots$, the first assertion of the lemma is proved.
Similarly, the second conclusion follows from the computation of $\min\{u(q+2h)-u(q+h), \ q \in [-0.12, 0.4]\} \approx0.02057.$ 
 \qed
\end{proof}
Since the inequality $u'(q+h)>0$ guarantees that the only critical point of $u(t)$ on the interval $[q,q+h]$ is a minimum point and 
$[0,1.316] \subseteq [0,\tilde f(0.105)]$, we obtain the following result: 
 \begin{corollary}\label{Cor31}
 For all $p \in [0,1.316]$, the solution $u(t,p)$ is $U$-shaped on the interval $(q, \nu^*(q))$ (see Figure \ref{F42}).  Moreover, $\mu(q)-q< 3\pi$ so that ${\cal R}'(p+)<0$ for all 
 $p \in (\tilde f(q_0),1.316]$, where $q_0\approx1.1845$ is computed in Example \ref{example24}. 
 \end{corollary}
 
  \begin{figure}[htb]
\centering {\includegraphics[width=6.9cm]{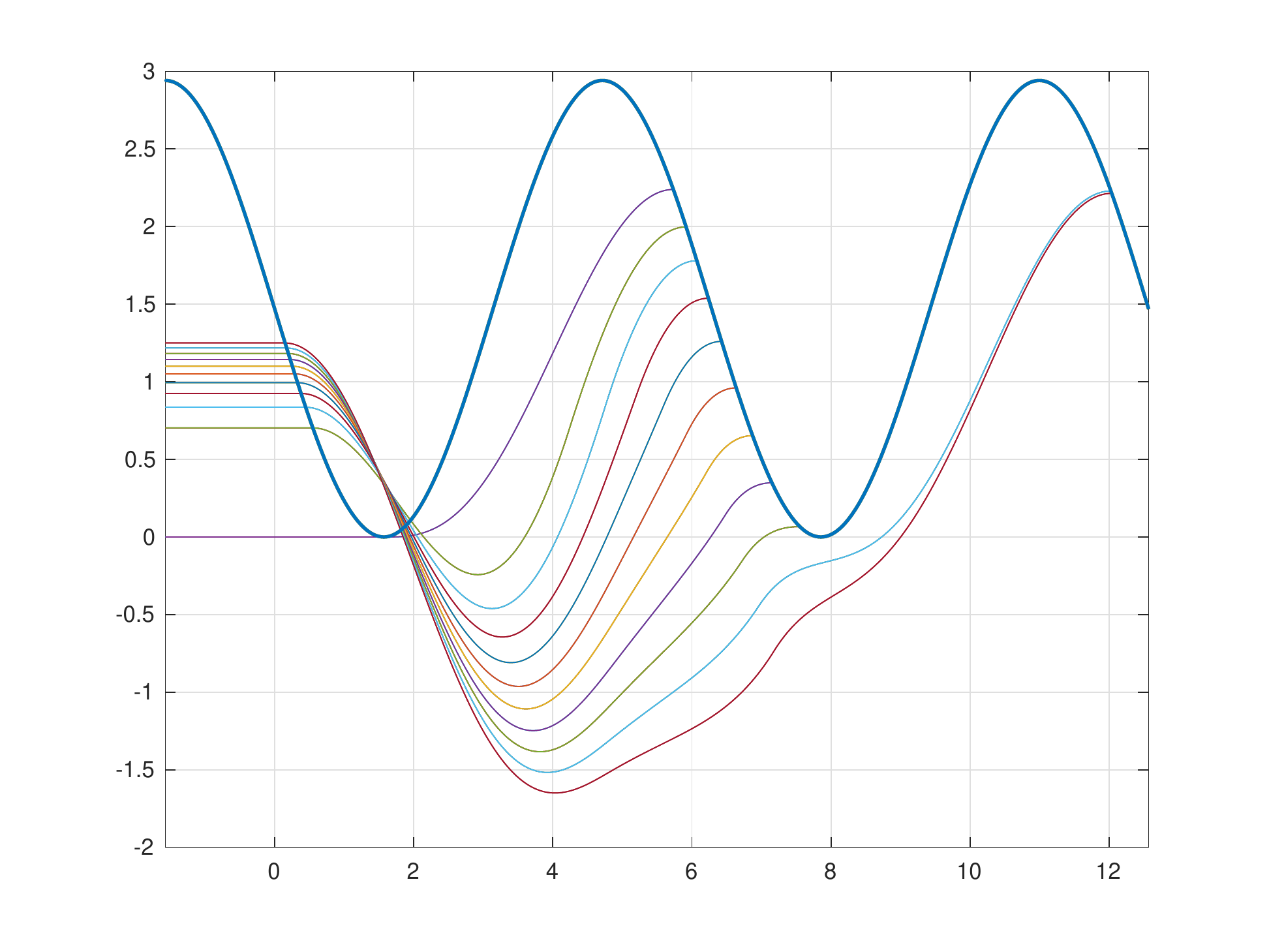}}{\includegraphics[width=6.9cm]{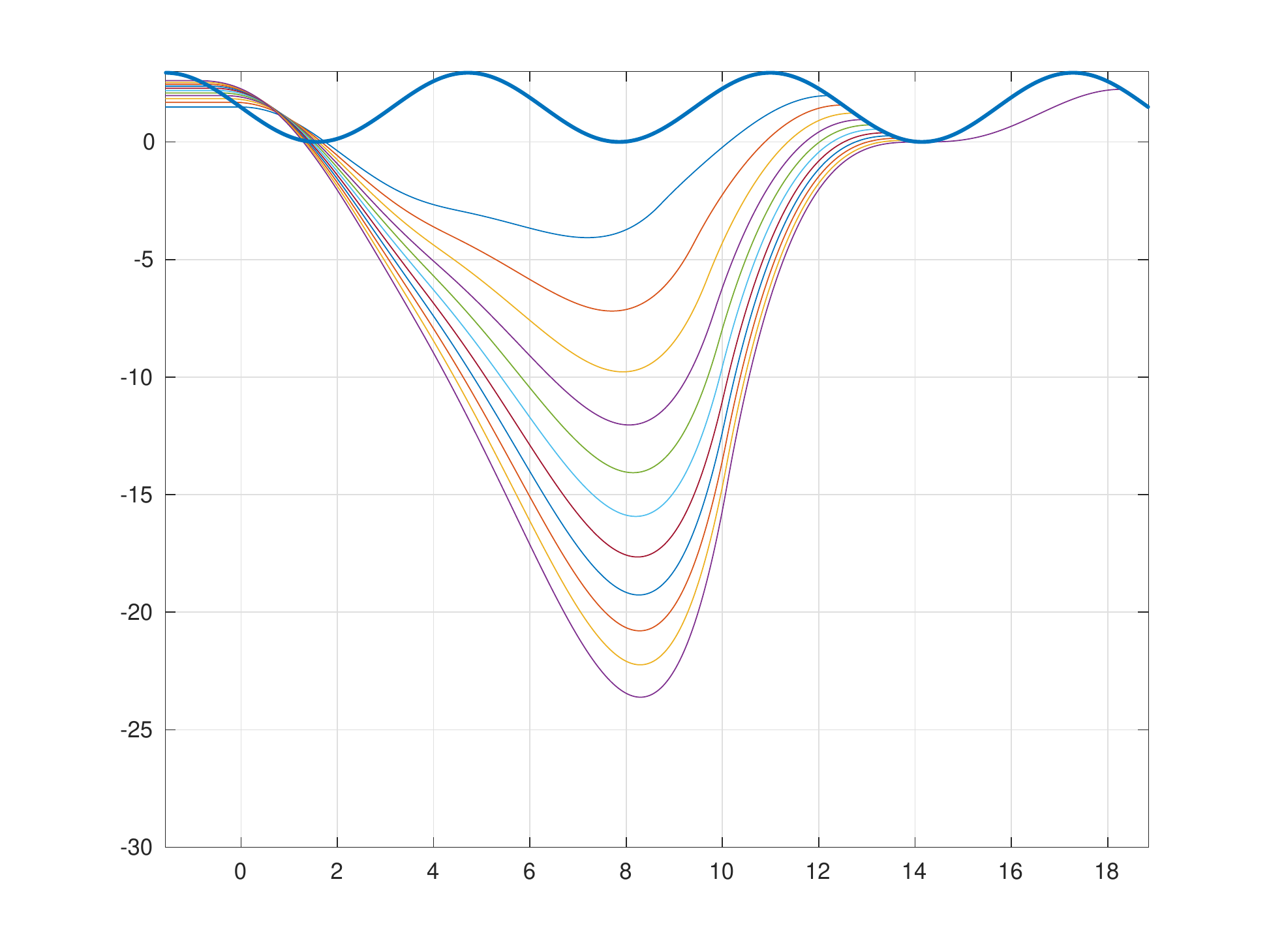}}
\caption{\hspace{0cm}  Graphs of solutions $u=u(t,p_j)$ on  $t \in [q_j,\nu^*(q_j)]$ together with the sine-like function $\tilde f$. On the left:   $p_j=\tilde f(q_j)=0.125\sqrt[4]{1000j}, \ j =0, \dots,10$.  On the right:  $p_j=\tilde f(q_j) =1+0.125\sqrt{15i}, \ i =1, \dots 11.$ } 
 \label{F42}
\end{figure} 
\begin{lemma}  \label{L32} The graph of 
 $u(t,p)$ does not contain the point $(2.5\pi,0)$ and   condition {\rm({\bf M})} is satisfied whenever $p \in {\cal I}= [1.26, 2/0.68]$. 
\end{lemma}
\begin{proof}  First, note that $u(t,p)=u(t), \ t \in [q,q+1.5\pi],$ for all $p \in {\cal I}$.  In particular, $u(t)$ has a unique critical point (global minimum 
point) on the interval $(q,q+1.5\pi]$  so that  $u(t) <0$ on $[\pi, q+1.5\pi]$ if $u(\pi) <0$ and $u(q+1.5\pi) =C_3 <0$. It is easy to check that these inequalities 
hold for all $q \in [-0.5\pi, 0.15]$. 

Next, for all $t\in [q+1.5\pi, q +3\pi]$, we find that 
$$
u'(t,p)= 0.32u(t,p)-U(t,p) +1-\sin t \leq 0.32u(t,p)-u(t-h,p) +1-\sin t. 
$$
Thus a standard comparison argument shows that $u(t,p) \leq u(t)$,   $t\in [q+1.5\pi, q +3\pi]$, where $u(t)$ is given by (\ref{sec}). 
Now, setting  $s=t-q-h\in [0,h]=[0,1.5\pi]$, we present $u(t)$ as
$$
\Phi(s,q)=  C^*_0\sin(s+q+2\theta_0)-C_0\cos(s+q+\theta_0)+C^*_1+\left[-C_2s + C^*_2\right]e^{as}. 
$$
Then the inequality $u(t)<0,\ t \in [q+1.5\pi, 2.5\pi],$ holds for all $q \in [-0.5\pi, 0.15]$ if 
$
\Phi(s,q) <0  
$
on the set $\Pi_2=\{(s,q): s+q\leq \pi, q \in [-0.5\pi, 0.15], s \geq 0\}$. Now,  we find that 
$$\max\{\Phi(s,q) , \ (s,q) \in \Pi_2\}= -0.0615\dots<0. $$
Thus $u(t,p) <0$ for all $t\in[\pi, 2.5\pi]$ whenever $q \in [-0.5\pi, 0.15]$. This proves  the first assertion of the lemma. 
Finally, since ${\cal R}(p)>0$,  condition ({\bf M}) is satisfied for each $p \in [\tilde f(0.15), \tilde f(-\pi/2)] = [1.25\dots, 2/0.68]$. \qed
\end{proof}

Now we are in a position to prove Theorem \ref{T24}.
\begin{proof}[of Theorem \ref{T24}] Since the computation of ${\cal R}(\bar p)$ for each given $\bar p = \tilde f(\bar q) \in K$ amounts to the explicit integration of  some first order inhomogeneous  linear differential equations with constant coefficients on a finite interval $[\bar q, \nu^*(\bar q)]$, and founding zeros of simple elementary functions on the same interval,  we will assume  that the value of ${\cal R}(\bar p)$ can be found with the required accuracy. For example,  the value of ${\cal R}(0)\approx2.2$ can be found by solving the equation $u(t)=\tilde f(t)$ on the interval $[1.5\pi, 2.5\pi]$, where $u(t)$ is given by (\ref{u1}). In a similar way, we can compute the value of  ${\cal R}({\cal R}(0)) \approx0.45.$

Next, Corollary \ref{Cor31} and Remark \ref{R27} allow to apply Theorem \ref{T25} on the $q$-interval  $(\alpha, \beta]=(0.105, 0.5\pi]$. In order to prove that the associated   $p$-interval  $[\tilde f(\beta), \tilde f(\alpha))=[0, 1.3\dots)$ contains  one point $p_1$ of discontinuity, it suffices to  take $q$ such that $\tilde f(q)=1.25$ and to check that  
$\nu^*(q) \in (3.5\pi,4.5\pi)$ (this $q$ corresponds to $q_{10}$ in   the left frame of  Figure~\ref{F42}). Invoking also Example \ref{example32},  we establish all stated properties of ${\cal R}$ on the interval $[0,1.316]$. Concerning the computation of the approximate value of $p_1$, note that  $p_1 \in (a_1,a_2) \subset [0,1.316]$ if ${\cal R}(a_1)< {\cal R}(a_2)$ (particularly, we obtain immediately that $p_1 \in (0.9,1.25)$ while a more accurate similar estimate implies that $p_1\approx1.1$). 

Finally,  Lemma \ref{L32} shows that condition ({\bf M}) is satisfied for all $p \in {\cal I}= [1.26, 2/0.68]$. Then Theorem \ref{T15} and  the proof of Corollary \ref{Cor20} imply that the restriction  ${\cal R}: [p_1,p_2)\to K$ has continuous graph 
until the  first eventual intersection of its closure with the real axis at some point $p_2$, where
${\cal R}(p_2 -) = 0, \ {\cal R}(p_2) = {\cal R}(0) >0$. In order to establish the existence of such $p_2$ and find its approximate value, 
it is enough to take $p^*_i=\tilde f(q_i^*) =1+0.125\sqrt{15k} \in \{2.53\dots, 2.60 \dots\},$  with $k=10,11$, and to note that $\nu^*(q_{11}^*)\in (5.5\pi,6.5\pi)$ while $\nu^*(q_{10}^*) \in (3.5\pi,4.5\pi)$, see  the right frame of Figure~\ref{F42}. In particular, this shows that $p_2>2.53 > {\cal R}(0)$. \qed

\end{proof}

\section*{Acknowledgments} We dedicate this work to the memory of  our colleague Anatoly Samoilenko (1938-2020),  
one of the most influential Soviet and Ukrainian experts in the field of ordinary differential equations (cf. \cite[Sections 2.43: V.I. Arnold and 2.49: A.M. Samoilenko]{MMM}) and beloved professor  and doctoral adviser of the first and third authors. In fact, our initial interest in model (\ref{meq}) was motivated by an approach to this equation based on Samoilenko's numerical-analytic method \cite{RST,SB}.

We express our appreciation to  Rafael Ortega for suggesting the present simple proof of Lemma \ref{L2}. We also thank   {\mL}ubom\'ir Snoha and Hans-Otto Walther for valuable discussions and suggestions. We are  indebted to Alexander Rezounenko for providing  the monograph \cite{Mago}, and to Hugo Huijer for his kind permission to reproduce his 2020 Happiness review,  whose original can be found in  \cite{th}. 

S. Trofimchuk  was   partially  supported by FONDECYT (Chile),   project 1190712, and   E. Liz by  the research
grant MTM2017--85054--C2--1--P (AEI/FEDER, UE).

\end{document}